\documentclass[11pt,reqno]{amsart}
\usepackage{graphicx}
\usepackage{amsmath}
\usepackage{mathtools}
\usepackage{amsfonts}
\usepackage{amssymb}
\usepackage{amsthm}
\usepackage{enumerate}
\usepackage{url}
\usepackage{kbordermatrix, mathdots, stmaryrd}
\usepackage{xypic}
\usepackage{mathdots}
\usepackage{algorithm}
\usepackage{algpseudocode}

\usepackage{tikz-cd}
\usepackage[left=1in, right=1in, top=1in, bottom=0.75in]{geometry}%

\providecommand{\U}[1]{\protect\rule{.1in}{.1in}}

\newtheorem{theorem}{Theorem}[section]
\newtheorem{proposition}[theorem]{Proposition}
\newtheorem{proposition/definition}[theorem]{Proposition/Definition}
\newtheorem{lemma}[theorem]{Lemma}
\newtheorem{corollary}[theorem]{Corollary}

\newtheorem{claim}[theorem]{Claim}

\theoremstyle{definition}
\newtheorem{definition}[theorem]{Definition}

\newtheorem{example}[theorem]{Example}

\newtheorem{question}[theorem]{Question}

\theoremstyle{remark}
\newtheorem{remark}[theorem]{Remark}

\begin{document}

\title{New Classes of Matrix Decompositions}

\author{Ke~Ye}
\thanks{KY is partially supported by AFOSR FA9550-13-1-0133, DARPA D15AP00109, NSF IIS 1546413, DMS 1209136, and DMS 1057064.}
\address{Department of Statistics, University of Chicago, Chicago, IL 60637-1514.}
\email{kye@galton.uchicago.edu}

\begin{abstract}
The idea of decomposing a matrix into a product of structured matrices such as triangular, orthogonal, diagonal matrices is a milestone of numerical computations. In this paper, we describe six new classes of matrix decompositions, extending our work in \cite{yelim}. We prove that every $n\times n$ matrix is a product of finitely many bidiagonal, skew symmetric (when n is even), companion matrices and generalized Vandermonde matrices, respectively. We also prove that a generic $n\times n$ centrosymmetric matrix is a product of finitely many symmetric Toeplitz (resp. persymmetric Hankel) matrices. We determine an upper bound of the number of structured matrices needed to decompose a matrix for each case.
\end{abstract}
\maketitle

\section{Introduction}
Matrix decomposition is an important technique in numerical computations. For example, we have classical matrix decompositions:
\begin{enumerate}
\item LU: a generic matrix can be decomposed as a product of an upper triangular matrix and a lower triangular matrix. 
\item QR: every matrix can be decomposed as a product of an orthogonal matrix and an upper triangular matrix.
\item SVD: every matrix can be decomposed as a product of two orthogonal matrices and a diagonal matrix.
\end{enumerate} 
These matrix decompositions play a central role in engineering and scientific problems related to matrix computations \cite{10},\cite{Decomp}. For example, to solve a linear system 
\[
A x = b
\]
where $A$ is an $n\times n$ matrix and $b$ is a column vector of length $n$. We can first apply LU decomposition to $A$ to obtain 
\[
LU x =b,
\] 
where $L$ is a lower triangular matrix and $U$ is an upper triangular matrix. Next, we can solve 
\begin{align*}
Ly &= b, \\
Ux & =y.
\end{align*}
to obtain the solution of the original linear equation. The advantage of decomposing $A$ into the product of $L$ and $U$ first is that solving linear equations with triangular matrix coefficient is much easier than solving the one with general matrix coefficient. Similar idea applies to QR and SVD decompositions. 

Those classical matrix decompositions (LU, QR and SVD decompositions) correspond to Bruhat, Iwasawa, and Cartan decompositions of Lie groups \cite{Knapp, Borel}. Other than those classical ones, there are other matrix decompositions. For instance, 
\begin{enumerate}
\item Every $n\times n$ matrix is a product of $(2n+5)$ Toeplitz (resp. Hankel) matrices \cite{yelim}.
\item Every matrix is a product of two symmetric matrices \cite{Bosch}.
\end{enumerate}
As we have seen for classical matrix decompositions, Toeplitz, Hankel and symmetric matrix decompositions are important in the sense that structured matrices are well understood. For example, a Toeplitz linear system can be solved in $O(n\log n)$ using displacement rank \cite{BA}, compared to at least $O(n^2)$ for general linear systems. Sometimes the matrix decomposition refers to the decomposition of a matrix into the sum of two matrices, see for example, \cite{ref1,ref2,ref3}. However, whenever we mention the matrix decomposition in this paper, we always refer to the multiplicative version.

In this article, we study matrix decompositions beyond those mentioned above. We use Algebraic Geometry as our tool to explore the existence of matrix decompositions for various structured matrices. We define necessary notions in Section \ref{sec:algebraic geometry} and we prove some general results for the matrix decomposition problem and establish a strategy to tackle the matrix decomposition problem in Section \ref{sec:general method}. In Section \ref{sec:toy example} we discuss the matrix decomposition problem with two factors and recover the LU decomposition and the QR decomposition for generic matrices using our method. Here the LU (resp. QR) decomposition for generic matrices means that the set of matrices which can be written as the product of a lower triangular (resp. orthogonal) matrix and an upper triangular matrix is a dense subset (with the Zariski topology) of the space of all $n\times n$ matrices.

In Section \ref{sec:linear spaces} we apply the strategy built in Section \ref{subsec:strategy} to matrix decomposition problem for linear subspaces. Lastly, in Section \ref{sec:nonlinear} we apply the strategy to matrix decomposition problem for non-linear varieties. We summarize our contributions in the following list:
\begin{enumerate}
\item Bidiagonal decomposition (Section \ref{subsec:bidiagonal}).
\item Skew symmetric decomposition (Section \ref{subsec:skew symmetric}).
\item Symmetric Toeplitz decomposition and persymmetric Hankel decomposition (Section \ref{subsec:symmetric Toeplitz})
\item Generic decomposition (Section \ref{subsec:generic}).
\item Companion decomposition (Section \ref{subsec:companion}).
\item Generalized Vandermonde decomposition (Section \ref{subsec:generalized Vandermonde}).
\end{enumerate}
For each type of matrices in the list above, we first prove the existence of the matrix decomposition for a generic matrix, in the sense that the set of all matrices which can not be decomposed as a product of matrices of the given type, is contained in a proper algebraic subvariety of $\mathbb{C}^{n\times n}$. Then we use a result from topological group theory to prove the existence of the matrix decomposition for every matrix. The price we need to pay is to increase the number of factors. Our method can only show the existence of the matrix decomposition and no algorithm can be obtained in general. However, we do find an algorithm for companion decomposition in \ref{subsec:companion}.

\section{Algebraic geometry}\label{sec:algebraic geometry}
In this section, we introduce necessary notions in Algebraic Geometry needed in this paper. We work over complex numbers but all results hold over algebraically closed fields of characteristic zero. Every notion we define in this section can be generalized to a more abstract version, but we only concentrate on a simplified version. Main references for this section are \cite{Taylor,shafarevich,Harris,GH,Hartshorne}

Let $\mathbb{C}[x_1,\dots,x_n]$ be the polynomial ring of $n$ variables over $\mathbb{C}$. We say that a subset $X\subset \mathbb{C}^n$ is an \textit{algebraic subvariety} if $X$ is the zero set of finitely many polynomials $F_1,\dots,F_r\in \mathbb{C}[x_1,\dots,x_n]$ and we say that $X$ is defined by $F_1,\dots, F_r$. For example, any linear subspace of $\mathbb{C}^n$ is an algebraic subvariety because they are all defined by polynomials of degree one. Less nontrivial examples are algebraic groups such as $\operatorname{GL}_n(\mathbb{C})$, the group of all $n\times n$ invertible matrices and $O(n)$, the group of all $n\times n$ orthogonal matrices. We remark here that $\operatorname{GL}_n(\mathbb{C})$ is an algebraic subvariety of $\mathbb{C}^{n^2+1}$ defined by 
\[
F(x_{ij},t) = t\operatorname{det}(x_{ij}) - 1\in \mathbb{C}[x_{ij},t],
\]
where $t,x_{ij},1\le i,j\le n$ are variables and $\operatorname{det}(x_{ij})$ is the determinant of the $n\times n$ matrix $(x_{ij})$. Also $O(n)$ is an algebraic subvariety of $\mathbb{C}^{n^2}$ because $O(n)$ is the defined by 
\[
(x_{ij})^{\mathsf{T}} (x_{ij}) - 1\in \mathbb{C}[x_{ij}],
\]
where $x_{ij},1\le i,j\le n$ are variables. 

Let $X$ be an irreducible algebraic subvariety of $\mathbb{C}^{n}$. We say that $X$ is \textit{irreducible} if $X$ cannot be written as the union of two algebraic subvarieties properly contained in $X$. In other words, whenever $X = X_1\cup X_2$ for algebraic subvarieties $X_1$ and $X_2$, we have 
\[
X_1 = X ~\text{or}~X_2 = X.
\]
It is clear that linear subspaces, $\operatorname{GL}_n(\mathbb{C})$ and $O(n)$ are all irreducible. The algebraic subvariety $X$ defined by the equation $x_1x_2 = 0$ is not irreducible since $X$ is the union of $X_i$ which is defined by the equation $x_i = 0,i=1,2$.

Let $X\subset \mathbb{C}^n$ be an algebraic subvariety. Let $I(X)$ be the ideal consisting of all polynomials $f\in \mathbb{C}[x_1,\dots,x_n]$ such that $f(x) = 0$ for any $x\in X$. It is well known \cite{shafarevich} that the ideal $I(X)$ must be finitely generated. Assume that $I(X)$ is generated by polynomials $F_1,\dots,F_r$. Let $J_p$ be the Jacobian matrix
\[
J_p = \begin{bmatrix}
\frac{\partial F_i}{\partial x_j}(p)
\end{bmatrix},
\]
where $1\le i\le r$ and $1\le j \le n$. We define the \textit{dimension} of $X$ to be
\[
\dim X \coloneqq \min_{p\in X} \{\dim \operatorname{ker}(J_p)\}. 
\]
The notion of dimension coincides with the intuition. For example, the dimension of a linear subspace of $\mathbb{C}^n$ is the same as its dimension as a linear space. The dimension of $\operatorname{GL}_n(\mathbb{C})$ is $n^2$ and the dimension of $O(n)$ is $\binom{n}{2}$. If $p\in X$ is a point such that 
\[
\dim X =  \operatorname{ker}(J_p),
\]
we say that $p$ is a \textit{smooth point} of $X$ and we define the \textit{tangent space} $T_pX$ of $X$ at $p$ to be 
\[
T_pX \coloneqq \operatorname{ker}(J_p).
\]
For example, the tangent space $T_p X$ of a linear subspace $X\subset \mathbb{C}^n$ at a point $p\in X$ can be identified with $X$ itself. The tangent space of $O(n)$ at the identity $e$ is simply the Lie algebra $\mathfrak{o}(n)$, the Lie algebra consisting of all $n\times n$ skew symmetric matrices \cite{FH}.

Let $U\subset \mathbb{C}^{n}$ be a subset. We define the \textit{Zariski closure} $\overline{U}$ of $U$ to be the common zero set of polynomials vanishing on $U$. Namely, 
\[
\overline{U} = \{x\in \mathbb{C}^n: f(x) = 0,f\in I(U) \}.
\]
For example, the Zariski closure of $\mathbb{R}$ in $\mathbb{C}$ is the whose space $\mathbb{C}$. We remark that the Zariski closure could be much larger than the Euclidean closure. For example, the Euclidean closure of $\mathbb{R}$ in $\mathbb{C}$ is just itself.

Let $X\subset \mathbb{C}^n$ and $Y\subset \mathbb{C}^m$ be two algebraic subvarieties. We say a map $f: X\to Y$ is a \textit{polynomial map} if $f$ can be represented as 
\[
f(x_1,\dots, x_n) = (f_1(x_1,\dots, x_n),\dots, f_m(x_1,\dots, x_n)),
\] 
where $f_1,\dots, f_m$ are polynomials in $n$ variables. For example, we denote $\mathbb{C}^{n\times n}$ by the space of all $n\times n$ matrices. It is clear that $\mathbb{C}^{n\times n} \cong \mathbb{C}^{n^2}$. Then the map 
\[
\rho_r: \underset{r \text{ copies }} {\underbrace{ \mathbb{C}^{n\times n}  \times \cdots \times \mathbb{C}^{n\times n} }} \to \mathbb{C}^{n\times n}
\]
defined by $\rho_r(A_1,\dots, A_r) = A_1\cdots A_r$ is a polynomial map. If $W_1,\dots, W_r$ are algebraic subvarieties of $\mathbb{C}^{n\times n}$, then the restriction of $\rho_r$ is, by definition, also a polynomial map.

We denote the set of all $k$ dimensional linear subspaces of $\mathbb{C}^n$ by $\operatorname{Gr}(k,n)$ and call it the \textit{Grassmannian of $k$ planes in $\mathbb{C}^n$}. In particular, when $k=1$, we have $\operatorname{Gr}(1,n) = \mathbb{P}^{n-1}$, the projective space. We say that a subset $X$ of $\mathbb{P}^{n-1}$ is a \textit{projective subvariety} if $X$ is defined by homogeneous polynomials, i.e., there are  homogeneous polynomials $F_1,\dots, F_r\in \mathbb{C}[x_1,\dots,x_n]$ such that
\[
X =\{[p]\in \mathbb{P}^{n-1}: F_1(p) = \cdots = F_r(p) = 0\}.
\]
Here $[p]$ is the element in $\mathbb{P}^{n-1}$, which corresponds to the line joining the origin and $p\in \mathbb{C}^n$. It is a fundamental fact \cite{shafarevich,Weyman,Harris,GH} that $\operatorname{Gr}(k,n)$ is a projective subvariety in $\mathbb{P}^{N-1}$ where $N = \binom{n}{k}$. Furthermore, $\operatorname{Gr}(k,n)$ is smooth, i.e., every point in $\operatorname{Gr}(k,n)$ is a smooth point, hence we can define the \textit{tangent bundle} $T\operatorname{Gr}(k,n)$, whose fiber over a point $[W]\in \operatorname{Gr}(k,n)$ is simply the tangent space $T_{[W]} \operatorname{Gr}(k,n)$. We define 
\[
E \coloneqq \{([W],w): [W]\in \operatorname{Gr}(k,n), w\in W\}
\] 
and a projection map $\pi:E \to  \operatorname{Gr}(k,n)$
\[
\pi([W],w) = [W].
\]
It turns out that $(E,\pi)$ is a vector bundle on $\operatorname{Gr}(k,n)$ \cite{shafarevich,Weyman,Harris,GH}. We say that $E$ is the \textit{tautological vector bundle on $\operatorname{Gr}(k,n)$}. By definition, the fiber $\pi^{-1}([W])$ over a point $[W]\in \operatorname{Gr}(k,n)$ is simply the vector space $W$.

Let $X$ be an algebraic subvariety in $\mathbb{C}^n$. Let $P$ be a property defined for points in $X$. We say that $P$ is a \textit{generic property} if there exists an algebraic subvariety $X_P\subsetneq X$ such that if $x\in X$ does not satisfy the property $P$, then $x\in X_P$. If the property $P$ is understood, we say that $x\in X$ is a \textit{generic point} if $x$ satisfies the property $P$. For example, we can say that for a fixed hyperplane $H$ in $\mathbb{C}^n$, a generic point in $x\in \mathbb{C}^n$ is not contained in $H$. A generic $n\times n$ matrix is invertible since matrices that are not invertible are defined by the vanishing of their determinants. We also say that for a fixed $m$ plane $L$, a generic $k$ plane intersects with $L$ in a $m+k-n$ dimensional subspace. If $x\in \mathbb{C}^n$ is a generic point for property $P$, then by definition, the set of points in $\mathbb{C}^n$ that does not satisfy $P$ is contained in an algebraic subvariety $X_P \subsetneq \mathbb{C}^n$. If we equip $\mathbb{C}^n$ with the Lebesgue measure, then it is clear that $X_P$ always has measure zero. In other words, the set of generic points has the full measure. 

In particular, when we say that a generic $n \times  n$ matrix can be decomposed into the product of finitely many structured matrices, for example, Toeplitz matrices, we mean the set of all matrices which admit such a decomposition is an open dense subset (with the Zariski topology) of the space of all $n\times n$ matrices. For those who are not familiar with the notion of generic objects, one can replace ``generic" by ``random" to obtain the intuition, though this is not rigorous in mathematics.

\section{General method}\label{sec:general method}
Let $X$ be an algebraic subvariety of $\mathbb{C}^{n\times n}$ which is closed under matrix multiplication, i.e., for any $A,B\in X$ we have $AB\in X$. Let $r$ be a positive integer and let $W_1,\dots, W_r$ be subvarieties of $X$. We consider a map 
\[
\rho_r: W^r \coloneqq W_1\times \cdots  \times W_r\to X
\]
defined by the matrix multiplication
\[
\rho_r(A_1,\dots,A_r) = A_1\cdots A_r.
\]
We can rephrase the matrix decomposition problem in terms of $\rho$. 
\begin{question}\label{question:main}~
\begin{enumerate}
\item Does there exist $r$ such that $\rho_r$ is a dominant map, i.e., $\overline{\rho_r(W^r)} = X$?
\item Does there exist $r$ such that $\rho_r$ is a surjective map, i.e., $\rho_r(W^r) = X$?
\end{enumerate}
\end{question}
Here the closure $\overline{\rho(W^r)}$ is the Zariski closure of $\rho(W^r)$ in $X$. In general, the fist question in Question \ref{question:main} is weaker than the second. However, we will see later that with some assumptions on $X$ and $W$, we can conclude that if $\rho_r$ is dominant, $\rho_{r'}$ is surjective for some $r'\ge r$. 

\subsection{Lower bound of $r$}
First, we can do a naive dimension counting to get a lower bound on $r$. To do this we need 
\begin{proposition}\label{prop:dimension counting}\cite{shafarevich}
If $f:Y \to Z$ is a dominant polynomial map between two irreducible algebraic varieties, then $\dim Y = \dim Z + \dim f^{-1}(z)$ for a generic point $z\in Z$. In particular, $\dim Y \ge \dim Z$.
\end{proposition}

Apply Proposition \ref{prop:dimension counting} to our case we obtain
\begin{corollary}\label{cor:dimension counting 1}
If $\rho_r: W^r \to X$ is dominant, then 
\begin{equation}\label{eqn:dimension counting}
\sum_{i=1}^r\dim W_i \ge \dim X.
\end{equation}
\end{corollary}
\begin{corollary}\label{cor:dimension counting 2}
If $\dim W_1= \cdots = \dim W_r=m $ and $\rho_r$ is dominant, then 
\[
r\ge \lceil \frac{\dim X}{m} \rceil.
\]  
\end{corollary}

We say that an algebraic subvariety $W\subset \mathbb{C}^{n\times n}$ is a \textit{cone} if $x\in W$ implies that $\lambda x\in W$ for any $\lambda \in \mathbb{C}$. Assume that $W_1,\dots, W_r\subset X$ are cones. For any $A\in X$ with a decomposition
\[
A = A_1\dots A_r,~\text{where}~ A_1\in W_1,\dots, A_r\in W_r,
\] 
$\rho_r^{-1}(A)$ contains the subvariety
\[
Z_A = \{(\lambda_1A_1,\dots, \lambda_rA_r)\in W^r: \lambda_i\in \mathbb{C}, \prod_{i=1}^r \lambda_i = 1\}.
\]
It is clear from the definition of $Z_A$ that
\[
\dim Z_A = r-1.
\]
Apply Proposition \ref{prop:dimension counting} to this case we obtain
\begin{corollary}\label{cor:dimension counting 3}
If $W_1,\dots, W_r$ are cones and $\rho_r: W^r \to X$ is dominant, then
\[
\sum_{i=1}^r \dim W_i - (r-1)\ge \dim X. 
\]
\end{corollary}
\begin{corollary}\label{cor:dimension counting 4}
If $W_1,\dots, W_r$ are cones of the same dimension $m$ and $\rho_r$ is dominant, then
\[
r\ge \lceil \frac{\dim X - 1}{m-1}\rceil.
\]
\end{corollary}

\subsection{Criterion for dominant maps}
\begin{proposition}\label{prop:dominant}
Let $r$ be an integer and $W_1,\dots,W_r$ be subvarieties of $X$. If there is a point $a =(A_1,\dots, A_r) \in W^r$ such that the differential 
\[
d\rho_r|_{a}: T_{A_1}W_1\times \cdots \times T_{A_r}W_r \to T_{\rho(a)}X
\]
has full rank $\dim X$, then the map $\rho_r$ is dominant.
\end{proposition}
\begin{proof}
Suppose that $\overline{\rho_r(W^r)}$ is not equal to $X$, then it is a proper subvariety of $X$ and hence it has dimension strictly smaller than that of $X$. Therefore, we have that the rank of $d\rho_r|_{a'}$ is strictly smaller than $\dim X$ for generic $a'\in W^r$. However, the assumption that there exists some $a\in W^r$ such that $d\rho_r|_{a}$ has the maximal rank implies that for a generic point $a'\in W^r$, we must have that the rank of $d\rho_r|_{a'}$ is equal to $\dim X$.
\end{proof}
For readers unfamiliar with the calculation of the differential $d\rho_r|_a$, we record the following formula
\begin{equation}
d\rho_r|_a(X_1,\dots, X_r) = \sum_{i=1}^r A_1\cdots A_{i-1} X_i A_{i+1} \cdots A_r,
\label{eqn:differential}
\end{equation}
where $X_i\in T_{A_i} W_i$. If in particular $W_i$ is a linear subspace of $\mathbb{C}^{n\times n}$, then we may identify the tangent space $T_{A_i}W_i$ as $W_i$ itself. We will apply formula \eqref{eqn:differential} repeatedly the the rest of this paper.

\subsection{Criterion for surjective maps}
\begin{proposition}\cite{Borel}\label{prop:group generation}
Let $G$ be a topological group and let $U$ be an open dense subset of $G$. Assume that $U$ contains the identity element of the group $G$. Then 
\[
G = U \cdot U,
\]
i.e, every element $g\in G$ is of the form $hh'$ for some $h,h'\in U$.
\end{proposition}

\begin{theorem}[open mapping theorem]\cite{Taylor}\label{thm:open mapping}
Let $X,Y$ be two irreducible varieties and let $f:X\to Y$ be a polynomial map. If $f$ is dominant then there is some $U\subset f(X)$ which is open and dense in $Y$.
\end{theorem}

\begin{proposition}\label{prop:surjective}
Let $W_1=\cdots =W_r =W$ be a linear subspace and $X = \mathbb{C}^{n\times n}$. Suppose that $W$ contains all diagonal matrices and that $\rho_r$ is dominant. Then the map
\[
\rho_{r'}: W^{r'}\coloneqq \underset{r'\text{ copies }} {\underbrace{W\times \cdots \times W}} \to \mathbb{C}^{n\times n}
\]
defined by matrix multiplication is surjective for $r' = 4r + 1$.
\end{proposition}
\begin{proof}
Since $\rho_r$ is dominant, by Theorem \ref{thm:open mapping} its image 
\[
\rho_r(W^{r})\coloneqq \underset{r\text{ copies }} {\underbrace{W\times \cdots \times W}} 
\]
contains an open dense subset of $\mathbb{C}^{n\times n}$. This implies that $\rho_r(W^{r})$ contains an open dense subset of the group $\operatorname{GL}_n(\mathbb{C})$ because $\operatorname{GL}_n(\mathbb{C})$ is an open dense subset of $\mathbb{C}^{n\times n}$. By Proposition \ref{prop:group generation} we see that 
\[
\operatorname{GL}_n(\mathbb{C}) \subseteq \rho_r(W^{r})\cdot \rho_r(W^{r}).
\]
Lastly, if $A\in \mathbb{C}^{n\times n}$ then there exists $P,Q\in \operatorname{GL}_n$ and a diagonal matrix $D\in \mathbb{C}^{n\times n}$ such that 
\[
A = P D Q.
\]
Hence we see that 
\[
\mathbb{C}^{n\times n} \subseteq \rho_r(W^{r})\cdot W \cdot \rho_r(W^{r})
\]
\end{proof}

\subsection{Our strategy}\label{subsec:strategy}
Let $X$ be an algebraic subvariety of $\mathbb{C}^{n\times n}$ which is closed under matrix multiplication and let $W_1,\dots, W_r$ be $r$ algebraic subvarieties of $X$. We define
\[
\rho_r: W^r \coloneqq W_1\times \cdots W_r \to X.
\]
In general, we may answer Question \ref{question:main} by the following strategy.
\begin{enumerate}
\item We first calculate the lower bound $r_0$ of $r$ for $\rho_r$ to be dominant according to Corollaries \ref{cor:dimension counting 2}, \ref{cor:dimension counting 3} and \ref{cor:dimension counting 4}. If $r< r_0$ then $\rho_r$ can not be dominant.
\item If $r\ge r_0$, we calculate the differential $d\rho_r|_a$ of $\rho_r$ at a point $a\in W^r$. If $d\rho_r|_a$ has the maximal rank $\dim X$ then $\rho_r$ is dominant by Proposition \ref{prop:dominant}.
\item If $W_1=\cdots = W_r = W$ is a linear subspace of $X=\mathbb{C}^{n\times n}$ containing all diagonal matrices and $\rho_r$ is dominant then by Proposition \ref{prop:surjective} $\rho_{r'}$ is surjective where $r' = 4r+ 1$.
\end{enumerate}
The main step in our strategy is to find a point $a\in W^r$ such that the differential of $\rho_r$ at $a$ has the rank $\dim X$. The rest of this paper is concentrating on applying the above strategy to answer Question \ref{question:main} for various choices of $W_i$ and $X$. 

\section{Toy examples: LU and QR decompositions}\label{sec:toy example}
In this section, we will discuss the matrix decomposition for two factors, which is the simplest case. We can recover the existence of LU and QR decompositions for a generic matrix using our method. 

We know that a generic matrix has
the LU decomposition and every matrix has the QR decomposition. Although the existence of the LU decomposition and the QR decomposition is
quit elementary and clear from the linear algebra point of view, it is
interesting to recover it from other point of view. In fact, we will prove a more general result.

\begin{theorem}\label{thm:triangular}
Let $W_1$ and $W_2$ be two algebraic subvarieties of $\mathbb{C}^{n\times n}$ such that
\begin{enumerate}
\item both $W_1$ and $W_2$ contain the identity matrix $I_n$ as a smooth point, and
\item $T_{I_n}W_1 + T_{I_n}W_2 = \mathbb{C}^{n\times n}$.
\end{enumerate}  
Then a generic $n\times n$ matrix is a product of some $A_1\in W_1$ and $A_2\in W_2$. Moreover, every $n\times n$ matrix $M$ is a product of some $A_1,B_1,C_1,D_1\in W_1$, $A_2,B_2,C_2,D_2\in W_2$ and a diagonal matrix $E$, i.e.,
\[
M = A_1B_1A_2B_2EC_1C_2D_1D_2.
\]
\end{theorem}
Here, by definition, we can identify the tangent space $T_{A_i} W_i$ at a smooth point $A_i$ with a linear subspace in $\mathbb{C}^{n\times n}$ for $i=1,2$.
\begin{proof}
Consider the differential
$d\rho_2|_{(I_n,I_n)}$ of $\rho_2$ at $(I_{n},I_{n})$, then by formula \eqref{eqn:differential} we must have 
\[
d\rho_2|_{(I_n,I_n)}(X_1,X_2) = X_1 + X_2,
\]
for any $X_i \in W_i,i=1,2$. By assumption we see that $d\rho_2|_{(I_n,I_n)}$ is surjective hence $\rho_2$ is dominant by Proposition \ref{prop:dominant}. The moreover part follows from Proposition \ref{prop:surjective}.
\end{proof}

Here we remind readers that by generic, we mean there are polynomials $F_1,\dots, F_r\in \mathbb{C}[x_{ij}]$ such that whenever $A = (a_{ij})$ is a matrix such that $A$ cannot be expressed as a product of some $A_1\in W_1$ and $A_2\in W_2$, then we must have 
\[
F_1(A) = \cdots = F_r (A) = 0.
\]
We warn readers that a generic matrix has a decomposition $A = A_1A_2$ for $A_i\in W_i,i=1,2$, does not imply that every matrix has such a decomposition. Intuitively speaking, a generic matrix has decomposition means ``most" matrices admit such a decomposition. For example, a generic matrix has the LU decomposition but it is not true that every matrix has the LU decomposition. We will see this phenomenon again in Section \ref{subsec:companion}, where we prove that a generic $n\times n$ matrix is a product of $n$ companion matrices, but there exits $n\times n$ matrices that do not admit such a decomposition.

\subsection{Triangular decomposition}\label{subsec:triangular}

We apply Theorem \ref{thm:triangular} to the LU decomposition and its variants.
\begin{corollary}\label{cor:triangular}
Let $A$ be a generic $n\times n$ matrix. Then

\begin{enumerate}
\item There exist a lower triangular matrix $L$ and an upper triangular matrix
$U$ such that
\[
A=LU.
\]

\item There exist a lower triangular matrix $L$ and an upper triangular matrix
$U$ such that
\[
A=UL.
\]

\item There exist a top triangular matrix $T$ and a bottom triangular matrix
$B$ such that
\[
A=TB.
\]

\item There exist a top triangular matrix $T$ and a bottom triangular matrix
$B$ such that
\[
A=BT.
\]

\end{enumerate}
\end{corollary}

\begin{remark}\label{remark:LU}
On the one hand, Corollary \ref{cor:triangular} does not specify what generic matrices are. It only guarantees that if we equip $\mathbb{C}^{n\times n}$ with Lebesgue measure, the probability that a randomly picked $n\times n$ matrix can be written as the product of a lower triangular matrix and an upper triangular matrix is one. On the other hand, It is known \cite{RJ} that a nonsingular matrix admit an LU decomposition if and only if all its leading principal minors are nonzero. This implies that a nonsingular matrix whose leading principal minors are all nonzero is a generic matrix in this case. However, we have more generic matrices, for example, matrices of rank $k$ whose first $k$ principal minors are nonzero are also generic matrices \cite{RJ}. It is also known \cite{O} that there exist $n\times n$ matrices which do not admit LU decompositions. Hence generic matrices for the LU decomposition form a proper subset of the space of $n\times n$ matrices.
\end{remark}

\begin{corollary}
Let $A$ be an $n\times n$ matrix. Then
\begin{enumerate}
\item There exist lower triangular matrices $L_{1},L_{2},L_{3},L_{4}$ and upper
triangular matrices $U_{1},U_{2},U_{3},U_{4}$ such that
\[
A=L_{1}U_{1}L_{2}U_{2}L_{3}U_{3}L_4U_4.
\]

\item There exist lower triangular matrices $L_{1},L_{2},L_{3},L_{4}$ and upper
triangular matrices $U_{1},U_{2},U_{3},U_{4}$ such that
\[
A=U_{1} L_{1} U_{2} L_{2} U_{3} L_{3}U_4L_4.
\]

\item There exist top triangular matrices $T_{1},T_{2},T_{3},T_4$ and bottom
triangular matrices $B_{1},B_{2},B_{3},B_4$ such that
\[
A=T_{1} B_{1} T_{2} B_{2} T_{3} B_{3}T_4B_4.
\]

\item There exist top triangular matrices $T_{1},T_{2},T_{3},T_4$ and bottom
triangular matrices $B_{1},B_{2},B_{3},B_4$ such that
\[
A=B_{1} T_{1} B_{2} T_{2} B_{3} T_{3}B_4T_4.
\]

\end{enumerate}
\end{corollary}

\subsection{QR decompositions}\label{subsec:QR} Assume $O(n)$ is the group of complex orthogonal matrices and let $U$ be
the space of upper triangular matrices. Since the tangent space of $O(n)$ at $I_n$ is simply the linear space of all $n\times n$ skew symmetric matrices \cite{FH} and both $O(n)$ and $U$ contain $I_n$ as a smooth point, we can apply Theorem \ref{thm:triangular} directly to $O(n)$ and $U$

\begin{corollary}\label{cor:QR}
Let $A$ be a generic $n\times n$ matrix. Then

\begin{enumerate}
\item There exist an orthogonal matrix $Q$, an upper triangular matrix $R$
such that
\[
A=RQ.
\]

\item There exist an orthogonal matrix $Q$, an upper triangular matrix $R$
such that
\[
A=QR.
\]

\item There exist an orthogonal matrix $Q$, a lower triangular matrix $S$ such
that
\[
A=QS.
\]

\item There exist an orthogonal matrix $Q$, a lower triangular matrix $S$ such
that
\[
A=SQ.
\]

\end{enumerate}
\end{corollary}

\begin{corollary}
\label{various decomposition} Let $A$ be an $n\times n$ matrix. Then

\begin{enumerate}
\item There exist orthogonal matrices $Q_{1},Q_{2},Q_{3},Q_{4}$ and upper
triangular matrices $R_{1},R_{2},R_{3},R_{4}$ such that
\[
A=R_{1}Q_{1}R_{2}Q_{2}R_{3}Q_{3}R_{4}Q_{4}.
\]

\item There exist orthogonal matrices $Q_{1},Q_{2},Q_{3},Q_{4}$ and upper
triangular matrices $R_{1},R_{2},R_{3},R_{4}$ such that
\[
A=Q_{1}R_{1}Q_{2}R_{2}Q_{3}R_{3}Q_{4}R_4.
\]

\item There exist orthogonal matrices $Q_{1},Q_{2},Q_{3},Q_{4}$ and lower
triangular matrices $S_{1},S_{2},S_{3},S_{4}$ such that
\[
A=Q_{1} S_{1} Q_{2} S_{2} Q_{3} S_{3} Q_{4} S_{4}.
\]

\item There exist orthogonal matrices $Q_{1},Q_{2},Q_{3},Q_{4}$ and lower
triangular matrices $S_{1},S_{2},S_{3},S_{4}$ such that
\[
A=S_{1}Q_{1}S_{2}Q_{2} S_{3}Q_{3}S_{4}Q_{4}.
\]
\end{enumerate}
\end{corollary}

\begin{remark}
Corollary \ref{various decomposition} is redundant because it is known that every $n\times n$ matrix admit a QR decomposition. Furthermore, this implies that generic matrices for QR decomposition in Corollary \ref{cor:QR} are actually all matrices. Combining this fact with Remark \ref{remark:LU}, we see that generic matrices are not necessarily the same for different matrix decompositions.
\end{remark}

One might ask whether or not the same method applies to the SVD, but
unfortunately, since the SVD involves complex conjugation of a matrix, which
makes the decomposition non-algebraic, we are not allowed to use the same
argument to recover the SVD even for generic matrices.

\section{Matrix decomposition for linear spaces}\label{sec:linear spaces}
\subsection{Bidiagonal decomposition}\label{subsec:bidiagonal}
Let $A = (a_{ij})$ be an $n\times n$ matrix. We say that $A$ is a \textit{$k$-diagonal matrix} if 
\[
a_{ij} = 0 ~\text{if}~\lvert i-j\rvert \ge k.
\]
In particular $1$-diagonal matrices are simply diagonal matrices, $2$-diagonal matrices are called bi-diagonal matrices. For example, a $3\times 3$ bidiagonal matrix is of the form 
\[
\begin{bmatrix}
a & b & 0\\
c & d & e \\
0 & f & g
\end{bmatrix}.
\]
An \textit{upper $k$-diagonal matrix} $A = (a_{ij})$ is a $k$-diagonal matrix with further restriction
\[
a_{ij} = 0 ~\textit{if}~ i-j\ge 1.
\]
A $3\times 3$ upper bidiagonal matrix is of the form
\[
\begin{bmatrix}
a & b & 0\\
0 & c & d \\
0 & 0 & e
\end{bmatrix}.
\]
A matrix $A$ is called \textit{lower $k$-diagonal} if $A^\mathsf{T}$ is upper $k$-diagonal. We denote the set of all $k$-diagonal matrices by $D_k$, the set of all upper $k$-diagonal matrices by $D_{k,\ge 0}$ and the set of all lower $k$-diagonal matrices by $D_{k,\le 0}$. 

\begin{lemma}\label{lemma: lower/upper k-diagonal}
Let $2\le k \le n$ be an integer. A generic $n\times n$ upper (resp. lower) $k$-diagonal matrix is a product of $k-1$ upper (resp. lower) bidiagonal matrices. In particular, A generic $n\times n$ upper (resp. lower) matrix is a product of $n$ upper (resp. lower) bidiagonal matrices.
\end{lemma}
\begin{proof}
We will prove the lemma for upper triangular matrix case. For a positive integer $2\le k$, we recall that $D_{k,\ge 0}$ is the space of upper $k$-diagonal matrices. It is clear that the product of $k-1$ bidiagonal matrices is a $k$ diagonal matrix. We want to show that the map
\[
\rho_{k-1}: \underset{k-1\text{ copies }} {\underbrace{D_{2,\ge 0}%
\times\cdots\times D_{2,\ge 0}}} \to D_{k,\ge 0}
\]
defined by matrix multiplication is dominant. We proceed by induction on $k$. When $k=2$, it is clear that $\rho_1$ is dominant. Assume that the map $\rho_{k-1}$ is dominant where $k-1\le n-1$ then we need to prove that $\rho_{k}$ is also dominant. To this end we can factor the map $\rho_k$ as 
\[
\rho_k: D_{2,\ge 0} \times \underset{k-1\text{ copies }} {\underbrace{( D_{2,\ge 0} \times \cdots\times D_{2,\ge 0})}} \xrightarrow{(\operatorname{Id}_n,\rho_{k-1})}  D_{2,\ge 0} \times D_{k-1,\ge 0} \xrightarrow{\rho} D_{k+1,\ge 0},
\]
where $\operatorname{Id}_n$ is the identity map on $D_{2,\ge 0}$ and $\rho$ is the map defined by multiplication of two matrices. By the induction hypothesis, we see that $(\operatorname{Id},\rho_{k-1})$ is dominant hence it is sufficient to show that $\rho$ is dominant. Now to see that $\rho$ is dominant, we calculate the differential of $\rho$. By formula \eqref{eqn:differential} the differential of $\rho$ at $(A,B)$ is given by
\[
d\rho|_{(A,B)}(X,Y)=XB+AY,~\text{for all}~(X,Y)\in D_{2,\ge0}\times
D_{k-1,\ge0}.
\]
On the other hand, given any $C\in D_{k+1,\ge0}$, we can write $C=C^{\prime
}+C^{\prime\prime}$ where $C^{\prime}\in D_{k,\ge0}$ and $C^{\prime\prime}\in
D_{k+1,\ge0}-D_{k,\ge0}$. We take $B=(\delta^{i}_{i+k-1})\in D_{k,\ge 0}$, where
$\delta^{i}_{j}$ is the Kronecker delta and $A=\operatorname{Id}_n$, then one can easily
find $(X,Y)\in D_{2,\ge0}\times D_{k,\ge0}$ such that
\[
XB=C^{\prime\prime}, Y=C^{\prime}.
\]
This implies that $d\rho|_{(A,B)}$ is surjective, and hence
$\rho$ is dominant by Proposition \ref{prop:dominant}.
\end{proof}

\begin{corollary}\label{cor:bidiagonal}
Every invertible upper (resp. lower) triangular matrix is a product of $2n$ upper (resp. lower) bidiagonal matrices.
\end{corollary}
\begin{proof}
Since a generic upper triangular matrix is a product of $n$ upper bidiagonal matrices, the corollary follows from Proposition \ref{prop:group generation}.
\end{proof}

\begin{proposition}\label{prop:bidiagonal}
A generic $n\times n$ matrix can be decomposed into a product of $2n$ tridiagonal matrices. An invertible $n\times n$ matrix is a product of $4n$ bidiagonal matrices.
\end{proposition}
\begin{proof}
By Lemma \ref{lemma: lower/upper k-diagonal} a generic upper (resp. lower) triangular matrix is a product of $n$ upper (resp. lower) bidiagonal matrices. A generic $n\times n$ matrix has an LU decomposition. Hence we see that a generic matrix can be decomposed as a product of $n$ upper and $n$ lower bidiagonal matrices. The last statement follows from \ref{prop:group generation}.
\end{proof}

\begin{theorem}\label{thm:bidiagonal} 
Let $r$ be the smallest number such that every $n\times n$ matrix is a product of $r$ bidiagonal matrices. Then
\[
n-1\le r\le 8n.
\]
\end{theorem}

\begin{proof}
Notice that every matrix $A$ can be written as 
\[
A = P D Q,
\]
where $P,Q$ are invertible and $D$ is diagonal. By Proposition \ref{prop:bidiagonal} we see that $P,Q$ are products of $8n$ bidiagonal matrices, respectively. Since diagonal matrices are also bidiagonal, we see that every $n\times n$ matrix is a product of $8n$ bidiagonal matrices. This gives the upper bound of $r$. For the lower bound, we simply notice that a product of $k-1$ bidiagonal matrices is $k$-diagonal hence $r$ must be at least $n-1$.
\end{proof}

Since $\dim D_{2}=3n-2$, by Corollary \ref{cor:dimension counting 4} the expected value of $r$ is $\lceil \frac{n+1}{3} \rceil$, while the lower bound of $r$ is $n-1$. This shows that Proposition \ref{thm:bidiagonal} gives us an example that the expected value may not be achieved. Roughly speaking, this is because entries on the diagonal of a tridiagonal matrix do not contribute to expand the product. To be more precise, if $X$ is a diagonal matrix and $Y$ is a bidiagonal matrix then $XY$ is still a bidiagonal matrix.

\subsection{Skew symmetric decomposition}\label{subsec:skew symmetric}
We consider skew symmetric matrix decomposition problem in this section. An $n\times n$ skew symmetric matrix $A$ is defined by the condition
\[
A = -A^{\mathsf{T}}.
\]
We denote the space of all $n\times n $ skew symmetric matrices by $\Lambda_n$. It is clear that $\Lambda_n$ is a linear subspace of $\mathbb{C}^{n\times n}$ and
\[
\dim (\Lambda_n) = \binom{n}{2}.
\]
On the one hand, since $\lceil \frac{n^2-1}{\binom{n}{2}-1} \rceil=3$ if $n\ge 3$, we see that if the map 
\[
\rho_r: \underset{r\text{ copies }}{\underbrace{\Lambda_n\times \cdots \times \Lambda_n}}\to \mathbb{C}^{n\times n}
\]
is dominant then $r$ is at least three. On the other hand, from the definition one can see that for any $A\in \Lambda_n$ we have
\[
\operatorname{det}(A) = \operatorname{det} (A^{\mathsf{T}} ) = (-1)^n \operatorname{det}(A).
\]
In particular if $n$ is odd, we obtain $\operatorname{det}(A)=0$. This implies that when $n$ is odd, the map $\rho$ can never be dominant, regardless how large $r$ is. However, when $n$ is odd, we can expect that 
\[
\rho_r: \underset{r\text{ copies }}{\underbrace{\Lambda_n\times \cdots \times \Lambda_n}}\to \operatorname{DET_n}
\]
is dominant for $r\ge 3$, where $\operatorname{DET_n}$ is the hypersurface of all $n\times n$ matrices whose determinants are zero.

\begin{proposition}\label{prop:skewsymmetric}
We have the following two cases:
\begin{enumerate}
\item $n$ is even. A generic $n\times n$ matrix is a product of $r$ skew symmetric matrices for $(n,r)$ where
\begin{enumerate}
\item $n\ge 8$, $r\ge 3$ or 
\item $n=6, r\ge 4$ or
\item $n=4, r\ge 5 $.
\end{enumerate}
\item $n$ is odd. A generic $n\times n$ matrix is a product of $r$ skew symmetric matrices whose determinants are zero for $(n,r)$ where 
\begin{enumerate}
\item $n\ge 5$, $r\ge 3$ or 
\item $n=3, r\ge 4$.
\end{enumerate}
\end{enumerate}
\end{proposition}
Again, we consider the map 
\begin{align*}
\rho_r: &\underset{r\text{ copies }}{\underbrace{\Lambda_n\times \cdots \times \Lambda_n}}\to \mathbb{C}^{n\times n},~\text{when $n$ is even, and}\\
\rho_r: &\underset{r\text{ copies }}{\underbrace{\Lambda_n\times \cdots \times \Lambda_n}}\to \operatorname{DET_n}, ~\text{when $n$ is odd}.
\end{align*}

\begin{example}
Using Macalay2 \cite{M2} we can calculate the dimension $d$ of the closure of the image of $\rho$ for small $n$ and $r$, we list some results
\begin{enumerate}
\item $n=2$, $d=1$ for any $r$.
\item $(n,r) = (3,3)$, $d=7$.
\item $(n,r) = (3,4)$, $d=8$, $\overline{\operatorname{im}(\rho_r)} = \operatorname{DET}_3$. 
\item $(n,r) = (4,3)$, $d=13$.
\item $(n,r) = (4,4)$, $d=15$, $\overline{\operatorname{im}(\rho_r)}$ is a hypersurface in $\mathbb{C}^{16}$.
\item $(n,r) = (4,5)$, $d=16$, $\overline{\operatorname{im}(\rho_r)}=\mathbb{C}^{16}$.
\item $(n,r)= (5,3), (7,3), (9,3)$ or $(11,3)$, $d=n^2-1$, $\overline{\operatorname{im}(\rho_r)} = \operatorname{DET}_n$.
\item $(n,r) = (6,3)$, $d=35$, $\overline{\operatorname{im}(\rho_r)}$ is a hypersurface in $\mathbb{C}^{36}$.
\item $(n,r) = (6,4)$, $d=36$, $\overline{\operatorname{im}(\rho_r)}=\mathbb{C}^{36}$.
\item $(n,r) = (8,3),(10,3),(12,3)$ or $(14,3)$, $d=n^2$, $\overline{\operatorname{im}(\rho_r)}=\mathbb{C}^{n^2}$.
\end{enumerate}
\end{example}
\begin{proof}[Proof of Proposition \ref{prop:skewsymmetric}]
It left to show that $\rho_r$ is dominant for $n\ge 16, r=3$ when $n$ is even (resp. $n\ge 13, r=3$ when $n$ is odd). We may proceed by induction on $n$. 
\begin{enumerate}
\item[Case $1$.] We assume $n\ge 16$ is even,
We consider three block diagonal matrices 
\[
\begin{bmatrix}
A & O \\
O & A'
\end{bmatrix},
\begin{bmatrix}
B & O \\
O & B'
\end{bmatrix}, \textit{ and }
\begin{bmatrix}
C & O \\
O & C'
\end{bmatrix},
\] 
where $A,B,C$ are $(n-8)\times (n-8)$ skew symmetric matrices, and $A',B',C'$ are $8 \times 8$ skew symmetric matrices such that differentials of 
\begin{align*}
\rho_3^{(n-8)} &: \Lambda_{n-8}\times \Lambda_{n-8} \times \Lambda_{n-8} \to \mathbb{C}^{(n-8) \times (n-8)} \textit{ and } \\
\rho_3^{(8)} &: \Lambda_{8}\times \Lambda_{8} \times \Lambda_{8} \to \mathbb{C}^{8 \times 8}
\end{align*}
are surjective at $(A,B,C)$ and $(A',B',C')$, respectively. We consider the differential of 
\[
\rho_3: \Lambda_{n}\times \Lambda_{n} \times \Lambda_{n} \to \mathbb{C}^{n \times n}
\]
at \[
a=
\begin{bmatrix}
A & O \\
O & A'
\end{bmatrix},
b=
\begin{bmatrix}
B & O \\
O & B'
\end{bmatrix}, \textit{ and }
c=\begin{bmatrix}
C & O \\
O & C'
\end{bmatrix}.
\]
We parametrize tangent spaces of $\Lambda_n$ at $a,b$ and $c$ by
\[
x=\begin{bmatrix}
X & u \\
-u^{\mathsf{T}} & X'
\end{bmatrix},
y=\begin{bmatrix}
Y & v \\
-v^{\mathsf{T}} & Y'
\end{bmatrix}, \textit{ and }
z=\begin{bmatrix}
Z & w \\
-w^{\mathsf{T}} & Z'
\end{bmatrix}.
\]
Then we have by formula \eqref{eqn:differential}
\begin{align*}
d{\rho_3}_{|(a,b,c)} (x,y,z)&= a b 
\begin{bmatrix}
Z & w \\
-w^{\mathsf{T}} & Z'
\end{bmatrix}
+ a 
\begin{bmatrix}
Y & v \\
-v^{\mathsf{T}} & Y'
\end{bmatrix}
c + 
\begin{bmatrix}
X & u \\
-u^{\mathsf{T}} & X'
\end{bmatrix}b c \\
& = 
\begin{bmatrix}
ABZ + AYC + XBC & ABw + AvC' + uB'C' \\
-A'B'w^{\mathsf{T}} - A'v^{\mathsf{T}}C - u^{\mathsf{T}} B C & A'B'Z' + A'Y'C' + X'B'C' 
\end{bmatrix}.  
\end{align*}
By choice of $A,B,C$ and $A',B',C'$ we know that $ABZ + AYC + XBC $ can be any $(n-8)\times (n-8)$ matrix and that $A'B'Z' + A'Y'C' + X'B'C'$ can be any $8\times 8$ matrix. Lastly, it is clear that $ABw + AvC' + uB'C'$ and -$A'B'w^{\mathsf{T}} - A'v^{\mathsf{T}}C - u^{\mathsf{T}} B C$ can be any $(n-8)\times 8$ and $8\times (n-8)$ matrix, respectively. This shows that ${d\rho_3}_{|(a,b,c)}$ is surjective hence $\rho$ is dominant.
\item[Case $2$.] We assume $n\ge 13$ is odd. Then we can choose 
\[
a=
\begin{bmatrix}
A & O \\
O & A'
\end{bmatrix},
b=
\begin{bmatrix}
B & O \\
O & B'
\end{bmatrix}, \textit{ and }
c=
\begin{bmatrix}
C & O \\
O & C'
\end{bmatrix}.
\]
where $A,B,C$ are $(n-5)\times (n-5)$ skew symmetric matrices and $A',B',C'$ are $5\times 5$ skew symmetric matrices such that 
\begin{align*}
\rho_3^{(n-5)} &: \Lambda_{n-5}\times \Lambda_{n-5} \times \Lambda_{n-5} \to \mathbb{C}^{(n-5) \times (n-5)} \textit{ and } \\
\rho_3^{(5)} &: \Lambda_{5}\times \Lambda_{5} \times \Lambda_{5} \to \operatorname{DET}_5
\end{align*}
are dominant at $(A,B,C)$ and $(A',B',C')$ respectively. The similar calculation as in the previous case shows that $\rho_3$ is dominant.
\end{enumerate}
\end{proof}
We remark that when $n=2$, a skew symmetric matrix is of the form 
\[
\begin{bmatrix}
0 & a \\
-a & 0
\end{bmatrix}, a\in\mathbb{C}.
\]
Therefore, we see that if $r$ is even, the image of $\rho_r$ is simply the space of all $2\times 2$ diagonal matrices, and if $r$ is odd, the image of $\rho_r$ is the space of skew symmetric matrices.

By Proposition \ref{prop:surjective} we can derive from \ref{prop:skewsymmetric} the following 
\begin{theorem} For $n\ge 4$, every $2n\times 2n$ matrix is a product of $13$ skew symmetric matrices. Every $6\times 6$ matrix is a product of $17$ skew symmetric matrices. Every $4\times 4$ matrix is a product of $21$ skew symmetric matrices.
\end{theorem}

Notice that we are not able to apply Proposition \ref{prop:surjective} when $n$ is odd. This is because the image of $\rho_r$ is contained in $\operatorname{DET}_n$, which does not contain the group of invertible matrices. 


\subsection{Symmetric Toeplitz matrix decomposition}\label{subsec:symmetric Toeplitz}
A symmetric Toeplitz matrix $A = (a_{ij})$ is defined by 
\[
a_{ij} = a_{i+p,j+p}, a_{ij} = a_{ji}, 1\le i, j, i+p, j+p \le n.
\]
We denote the space of all symmetric Toeplitz matrices by $ST_n$. A centrosymmetric matrix $B = (b_{ij})$ is defined by 
\[
b_{ij} = a_{n-i+1, n-j+1}, 1\le i,j \le n.
\]
It is easy to verify that the product of two centrosymmetric matrices is again a centrosymmetric matrix hence the space $CS_n$ of all centrosymmetric matrices is an algebra. Moreover, we have 
\[
ST_n \subset CS_n.
\]
We say that a matrix $A$ is a persymmetric Hankel if $JA$ is a symmetric Hankel, where 
\[
J = \begin{bmatrix}
0 & 0 & \cdots & 0 & 1\\
0 & 0 & \cdots & 1 & 0\\
\vdots & \vdots & \ddots & \vdots & \vdots\\
0 & 1 & \cdots & 0 & 0\\
1 & 0 & \cdot & 0 & 0
\end{bmatrix}.
\]
We denote the space of all $n\times n$ persymmetric Hanekl matrices by $PH_n$. It is clear that $PH_n \subset CS_n$.

We will consider symmetric Toepliz (resp. persymmetric Hankel) matrix decomposition problem of a centrosymmetric matrix. It is clear that 
\[
\dim(ST_n) = \dim(PH_n) = n, \dim(CS_n) = \lceil \frac{n^2}{2} \rceil.
\]
Hence by Corollary \ref{cor:dimension counting 4} we see that if 
\[
\rho_r: \underset{r~\text{copies}} {\underbrace{ST_n\times \cdots \times ST_n}} \to CS_n\quad \text{or}
\]
\[
\rho_r: \underset{r~\text{copies}} {\underbrace{PH_n\times \cdots \times PH_n}} \to CS_n
\]
is dominant, then 
\[
r\ge \frac{\lceil \frac{n^2}{2} \rceil - 1}{n - 1} = \begin{cases}
\lfloor \frac{n+1}{2} \rfloor,&~\text{if}~n\ge 3, \\
1,&~\text{if}~n=2.
\end{cases}
\]

\begin{proposition}\label{prop: centrosymmetric}
Let $n\ge 3$ be an integer. A generic $n\times n$ centrosymmetric matrix is a product of $\lfloor \frac{n+1}{2}\rfloor $ symmetric Toepitz (resp. persymmetric Hankel) matrices.
\end{proposition}
The proof of Proposition \ref{prop: centrosymmetric} is similar to the proof of Toeplitz matrix decomposition theorem \cite{yelim} hence we will just give a sketch of the proof for Proposition \ref{prop: centrosymmetric} here. 
\begin{proof}[Sketch of proof of Proposition \ref{prop: centrosymmetric}]
It is sufficient to prove the statement for symmetric Toeplitz. Indeed, since we have 
\[
J A = A J, J^2 = 1.
\]
if $A$ is symmetric Toeplitz, if $X\in CS_n$ has a decomposition
\[
X = A_1\cdots A_r
\]
where $A_1,\dots, A_r\in ST_n$ and $r = \lfloor \frac{n+1}{2}\rfloor$. Then we see that 
\[
JX = J(A_1\cdots A_r) = (JA_1) \cdots (JA_r),
\]
if $r$ is odd and 
\[
X = (JA_1) \cdots (JA_r),
\]
if $r$ is even. In either case, this implies that a generic centrosymmetric matrix is a product of $r$ persymmetric Hankel matrices. 

Let $B_k\coloneqq (\delta_{i,j+k})^n_{i,j=1}, k=-(n-1),-(n-2),\dots, n-1$ be a basis of the linear space of all $n\times n$ Toeplitz matrices. Then $S_k\coloneqq \frac{1}{1+ \delta_{k.-k}}(B_k + B_{-k}), k =0,1, \dots, n-1$ is a basis of $ST_n$. Precisely, we have 
\[
B_{0}=\left[
\begin{smallmatrix}
1 & 0 & 0 &  & \\
0 & 1 & 0 & \smash{\ddots} & \\
0& 0  & 1 & \smash{\ddots}  & 0\\
&  \smash{\ddots} & \smash{\ddots}& \smash{\ddots} & 0\\
&  & 0  & 0  & 1
\end{smallmatrix} \right],
B_{1}=\left[
\begin{smallmatrix}
0 & 1 & 0 &  & \\
& 0 & 1 & \smash{\ddots} & \\
&  & 0 & \smash{\ddots}  & 0\\
&  &  & \smash{\ddots} & 1\\
&  &  &  & 0
\end{smallmatrix} \right],
B_{2}=\left[
\begin{smallmatrix}
0 & 0 & 1 &  & \\
& 0 & 0 & \smash{\ddots} & \\
&  & 0 & \smash{\ddots}  & 1\\
&  &  & \smash{\ddots} & 0\\
&  &  &  & 0
\end{smallmatrix} \right],
\dots,B_{n-1}=\left[
\begin{smallmatrix}
0 & 0 & 0 &  &1 \\
& 0 & 0 & \smash{\ddots} & \\
&  & 0 & \smash{\ddots}  & 0\\
&  &  & \smash{\ddots} & 0\\
&  &  &  & 0
\end{smallmatrix}\right],
\]
and 
\[
S_{0}=\left[
\begin{smallmatrix}
1 & 0 & 0 &  & \\
0 & 1 & 0 & \smash{\ddots} & \\
0& 0  & 1 & \smash{\ddots}  & 0\\
&  \smash{\ddots} & \smash{\ddots}& \smash{\ddots} & 0\\
&  & 0  & 0  & 1
\end{smallmatrix} \right],
S_{1}=\left[
\begin{smallmatrix}
0 & 1 & 0 &  & \\
1& 0 & 1 & \smash{\ddots} & \\
0& 1 & 0 & \smash{\ddots}  & 0\\
& \smash{\ddots} &  \smash{\ddots}& \smash{\ddots} & 1\\
&  & 0 & 1 & 0
\end{smallmatrix} \right],
S_{2}=\left[
\begin{smallmatrix}
0 & 0 & 1 &  & \\
0& 0 & 0 & \smash{\ddots} & \\
1& 0 & 0 & \smash{\ddots}  & 1\\
& \smash{\ddots} & \smash{\ddots} & \smash{\ddots} & 0\\
&  & 1 &  0& 0
\end{smallmatrix} \right],
\dots,S_{n-1}=\left[
\begin{smallmatrix}
0 & 0 & 0 &  &1 \\
0 & 0 & 0 & \smash{\ddots} & \\
0 & 0 & 0 & \smash{\ddots}  & 0\\
&  \smash{\ddots} &  \smash{\ddots} & \smash{\ddots} & 0\\
1&  & 0 & 0 & 0
\end{smallmatrix}\right],
\]

To prove that $\rho_r$ is dominant, it suffices to find a point $a\coloneqq (A_{n-r},\dots, A_{n-1})\in \underset{r~\text{copies}}{\underbrace{ST_n\times \cdots \times ST_n}}$ such that the differential of $\rho_r$ at $a$ has the maximal rank $\lceil \frac{n^2}{2}\rceil$. In stead of choosing a point $a$ explicitly, we will show that such a point exits. To this end, we write 
\[
A_{n-i} \coloneqq S_0 + t_{n-i} S_{n-i}, i=1,2,\dots, r,
\] 
where $t_{n-1},\dots,t_{n-r} $ are indeterminants. For such $A_{n-i}$ we see that the differential $d\rho_r |_{a}$ can be represented as an $\lceil \frac{n^2}{2}\rceil \times rn$ matrix $M$, whose entries are polynomials in $t_{n-1},\dots, t_{n-r}$. Now to see that $M$ has rank $\lceil \frac{n^2}{2}\rceil$, we need to find a nonzero $\lceil \frac{n^2}{2}\rceil\times \lceil \frac{n^2}{2}\rceil$ minor of $M$. 

\begin{claim}\label{claim:centrosymmetric} By the same type of argument as in \cite{yelim}, we can show 
\begin{enumerate}
\item Any $\lceil \frac{n^2}{2}\rceil\times \lceil \frac{n^2}{2}\rceil$ of $M$ is a polynomial in $t$'s of degree at least $\lceil \frac{n^2}{2}\rceil- n$.
\item There exists a $\lceil \frac{n^2}{2}\rceil\times \lceil \frac{n^2}{2}\rceil$ minor of $M$ contains a monomial of degree exactly $\lceil \frac{n^2}{2}\rceil- n$ and whose coefficient is non-zero. Indeed, the desired monomial is $t_{n-1}^{n-1}t_{n-2}^{n-1}\cdots t_{n-r+1}^{n-1}$ if $n$ is odd, and $t_{n-1}^{n-1}t_{n-2}^{n-1}\cdots t_{n-r+2}^{n-1}t_{n-r+1}^{\frac{n}{2}-1}$ if $n$ is even.
\end{enumerate}
\end{claim}
This shows that for a fixed $n$, there exist some values of $t_{n-1},\dots, t_{n-r}$ such that the differential $d\rho_r|_{a}$ has the maximal rank $\lceil \frac{n^2}{2}\rceil$.
\end{proof}

We work out an example to illustrate how the proof of Proposition \ref{prop: centrosymmetric} works. We adopt notations in \cite{yelim}. For the map 
\[
\rho_r: \underset{r~\text{copies}} {\underbrace{ST_n\times \cdots \times ST_n}} \to CS_n,
\]
we define 
\[
X_{n-i} \coloneqq \sum_{k=0}^{n-1} x_{n-i,k} S_k, i=1,2,\dots, r
\]
to be the matrix occuring in the $i$-th argument of $\rho_r$. Then by formula \eqref{eqn:differential}, the differential $d\rho_r|a$ is simply a linear map defined by
\[
d\rho_r|a (X_{n-r},\dots,X_{n-1}) = \sum_{i=1}^r A_{n-r} \cdots A_{n-i-1} X_{n-i} A_{n-i+1}\cdots A_{n-1},
\]
for $X_{n-i}\in ST_n$. We denote the entries of $d\rho_r|a (X_{n-r},\dots,X_{n-1})$ by $L_{p,q},1\le p,q\le n$. Then it is clear that the matrix $(L_{p,q})$ is a centrosymmetric matrix, i.e, 
\[
L_{p,q} = L_{n-p+1,n-q+1}, 1\le p,q \le n.
\]
We 
\begin{enumerate}
\item[Case $1$.] We consider the case $n=3$. This gives 
\[
r = \lfloor \frac{n+1}{2} \rfloor = 2, \lceil \frac{n^2}{2} \rceil = 5.
\]
We will see that any $5 \times 5$ minor of the $5\times 6$ matrix $M$ is a polynomial in $t$'s of degree at least $3$. We can simply write 
\[
L_{p,q} = x_{1,q-p} + x_{2,q-p} + \Omega(t), p,q=1,2,3.
\]
Here $\Omega(t)$ stands for terms of $L_{p,q}$ of degree at least one in $t$'s. With this notation, we express $M$ as 
\[\arraycolsep=2.5pt
\kbordermatrix{
& x_{1,2} & x_{2,2} &  x_{1,1} & x_{1,0} &x_{2,0} & x_{2,1}  \\
L_{3,3} = L_{1,1} & * & * & * & 1&  1& *  \\
L_{3,2} = L_{1,2} & * & * & 1 & *& * & 1  \\
L_{3,1} = L_{1,3} & 1 & 1 & * &* &*  &*   \\
L_{2,2}                & * &*  &*  &1 &1  &*  \\
L_{2,1} = L_{23}  & * &*  &1  &* &*  &1   }.
\]
Here $\*$ means that the entry if of the form $\Omega(t)$ and $1$ means that the corresponding $L_{p,q}$ contains $x_{1,k}$ or $x_{k}$. For example, since $L_{3,3} = L_{1,1}$ is of the form 
\[
x_{1,0} + x_{2,0} + \Omega(t),
\]
we put $1$ in $(1,4)$-th and $(1,5)$-th entry of $M$ and $*$ elsewhere in the first row. It is not hard to see that any $5\times 5$ minor of $M$ has degree at least $5-3 =2$. This verifies the first statement of Claim \ref{claim:centrosymmetric}.
\item[Case $2$.] Next, we consider the case $n=5$. In this case we have
\[
r=\lfloor \frac{n+1}{2} \rfloor = 3, \lceil \frac{n^2}{2} \rceil = 13.
\]
We consider the table 
\[
\begin{matrix}
L_{i,j} & 1 & 2 & 3 & 4 & 5 \\\hline
1 & t_{4} & t_{3} & 1 & 1 & 1\\
2 & t_{4} & t_{3} & 1 & t_{3} & t_{4}\\
3 &  &  & 1 & t_{3} & t_{4}
\end{matrix},
\]
indicating the way we obtain $t_4^4t_3^4$. Namely, we pick $t_4$ from $L_{1,1},L_{2,1},L_{2,5}$ and $L_{3,5}$. We pick $t_3$ from $L_{1,2},L_{2,3},L_{2,4}$ and $L_{3,4}$ and one form rest five $L_{ij}$'s. By definition of $L_{ij}$, we see that this is the unique way to obtain the monomial $t_4^4t_3^4$. This verifies the second statement of Claim \ref{claim:centrosymmetric}.
\end{enumerate}

\subsection{Generic matrix decomposition}\label{subsec:generic}
In this section we consider the decomposition problem for generic linear subspaces of $\mathbb{C}^{n\times n}$. Let $r$ be a positive integer. Assume that for $i=1,2,\dots,r$, $W_i$ is a $k_i$ dimensional subspace of $\mathbb{C}^{n\times n}$. We define 
\[
W^r \coloneqq {W_1\times \cdots \times W_r}.
\]
Let $\rho_{r}:W^{r}\rightarrow\mathbb{C}^{n\times n}$ be the map sending
$(A_{1},\dots,A_{r})$ to their product $A_1\cdots A_r$. We want to determine $r$, such
that $\rho_{r}$ is dominant. Consider the following diagram:
\[
\xymatrix{
TE_1\times \cdots \times \ TE_r\ar[d] \ar[r]^-{d{\hat{\rho_{r}}}}  &  T\operatorname{Gr}(k_1,n^{2})\times \cdots \times T\operatorname{Gr}(k_r,n^2)\times \mathbb{C}^{n\times n}\ar[d]\\
E_1\times \cdots \times E_r \ar[d] \ar[r]^-{\hat{\rho_r}}  & \operatorname{Gr}(k_1,n^{2})\times \cdots \times \operatorname{Gr}(k_r,n^2)\times\mathbb{C}^{n\times n}\ar[d]\\
\operatorname{Gr}(k_1,n^{2})\times \cdots \times \operatorname{Gr}(k_r,n^2)\ar[r]^-{\operatorname{id}}  &  \operatorname{Gr}(k_1,n^{2})\times \cdots \times \operatorname{Gr}(k_r,n^2).
}
\]
Here $E_i$ is the tautological vector bundle over
$\operatorname{Gr}(k_i,n^{2})$, $TE_i$ is the tangent bundle of $E_i$, and $\hat{\rho_{r}}$
is the bundle map induced by $\rho_{r}:W^{r}\rightarrow\mathbb{C}^{n^{2}}$,
and $d\hat{\rho_{r}}$ is the differential of $\hat{\rho_{r}}$. 

To be more precise, for any $[W_i]\in\operatorname{Gr}(k_i,n^{2})$, the fiber $E_{[W_i]}$ over $[W_i]$ is $W_i$
and the fiber of $TE_i$ over $([W_i],A_i)$ is $T_{[W_i]}%
\operatorname{Gr}(k_i,n^{2})\oplus W_i$, where $A_i\in W_i$. If we restrict $\hat{\rho_{r}}$
to the fiber $E_{[W_1]}\times \cdots \times E_{[W_r]}$ we obtain the map $\rho_{r}:W^{r}\rightarrow
\mathbb{C}^{n\times n}$ defined before and the restriction of $d\hat{\rho_{r}}$ to ${TE_1}_{([W_1],A_{1})}\times \cdots \times {TE_r}_{([W_r],A_r)}$ becomes the differential of
$\rho_{r}$ at the point $(A_{1},\dots,A_{r})$.

\begin{lemma}\label{lem:generic}
Let $r$ be a positive integer. For each $i=1,\dots, r$, let $k_i$ be a fixed integer such that $1\leq k_i\leq n^{2}$. Assume
that $\rho_{r}:W^{r}\rightarrow\mathbb{C}^{n\times n}$ is dominant for some
$k_i$ dimensional subspace $W_i$ of $\mathbb{C}^{n\times n}$, $i=1,2,\dots, r$. Then for a generic
$k_i$-dimensional subspace $W_i^{\prime}$ of $\mathbb{C}^{n\times n}$, the map
$\rho_{r}:W^{\prime r}\rightarrow\mathbb{C}^{n\times n}$ is dominant, where 
\[
{W^\prime}^r \coloneqq {{W^\prime}_1\times \cdots \times {W^\prime}_r}.
\]
\end{lemma}

\begin{proof}
Since $\rho_r: W^r \to \mathbb{C}^{n\times n}$ is dominant, we see that the Jacobian matrix of $\rho_r$ at a generic point $(A_1,\dots, A_r)$ in $W^r$ has the maximal rank, i.e., the Jacobian matrix of $\hat{\rho_r}$ at $([W_1],A_1)\times \cdots \times ([W_r],A_r)$ has the maximal rank. By Proposition \ref{prop:dominant} we see that $\hat{\rho_r}$ is dominant, i.e., $\rho_r$ is dominant for generic ${W^\prime}_i\in \operatorname{Gr}(k_i,n^2)$, $i=1,2,\dots, r$.
\end{proof}

We shall make use of the following result

\begin{proposition}\cite{yelim}\label{prop:toeplitz}
A generic $n\times n$ matrix can be decomposed as the product of $\lfloor \frac{n}{2}\rfloor + 1$ Toeplitz matrices.
\end{proposition}

\begin{proposition}\label{prop:generic}
Let $\mathbb{C}^{n\times n}$ be the space of all $n\times n$ matrices then

\begin{enumerate}
[\upshape (i)]

\item For generic $(2n-1)$ dimensional subspaces $W_1,\dots, W_r$ of $\mathbb{C}^{n\times
n}$, a generic $n\times n$ matrix is a product of elements in $W_i,i=1,2,\dots,r$ if $r\geq\lfloor \frac{n}{2}\rfloor + 1$.

\item For generic
$\binom{n+1}{2}$ dimensional subspaces $W_1,\dots, W_r$ of $\mathbb{C}^{n\times n}$, a generic $n\times n$ matrix is a product of elements in $W_i,i=1,2,\dots,r$ if $r\geq 2$.

\item For generic $(3n-2)$ dimensional subspaces $W_1,\dots, W_r$ of $\mathbb{C}^{n\times n}$,
a generic $n\times n$ matrix is a product of elements in $W_i,i=1,2,\dots,r$ if $r\geq 2n$.

\item For generic $\binom{2n}{2}$ dimensional subspaces $W_1,\dots, W_r$ of $\mathbb{C}^{2n\times 2n}$, a generic $n\times n$ matrix is a product of elements in $W_i,i=1,2,\dots,r$ if 
\begin{enumerate}
\item $r\geq 3$ when $n\ge 4$,
\item $r\ge 4$ when $n=3$,
\item $r\ge 5$ when $n=2$.
\end{enumerate}

\end{enumerate}
\end{proposition}
\begin{proof}
The first statement follows from Proposition~\ref{prop:toeplitz} and Lemma~\ref{lem:generic}.
The second statement follows from Corollary~\ref{cor:triangular} and Lemma~\ref{lem:generic}. The third statement follows from Proposition \ref{prop:bidiagonal} and Lemma \ref{lem:generic}. The last statement follows from Proposition \ref{prop:skewsymmetric} and Lemma Lemma \ref{lem:generic}.
\end{proof}

The combination of Proposition \ref{prop:surjective} and Proposition \ref{prop:generic} implies 
\begin{theorem}\label{thm:generic}
Let $\mathbb{C}^{n\times n}$ be the space of all $n\times n$ matrices then

\begin{enumerate}
[\upshape (i)]

\item For generic $(2n-1)$ dimensional subspaces $W_1,\dots, W_r$ of $\mathbb{C}^{n\times
n}$, an $n\times n$ matrix is a product of elements in $W_i,i=1,2\dots,r$ if $r\geq 2n+5$.

\item For generic
$\binom{n+1}{2}$ dimensional subspaces $W_1,\dots, W_r$ of $\mathbb{C}^{n\times n}$, an $n\times n$ matrix is a product of elements in $W_i,i=1,2\dots,r$ if $r\geq 9$.

\item For generic $(3n-2)$ dimensional subspaces $W_1,\dots, W_r$ of $\mathbb{C}^{n\times n}$,
an $n\times n$ matrix is a product of elements in $W_i,i=1,2\dots,r$ if $r\geq 8n+1$.

\item For generic $\binom{2n}{2}$ dimensional subspaces $W_1,\dots, W_r$ of $\mathbb{C}^{2n\times 2n}$,
an $n\times n$ matrix is a product of elements in $W_i,i=1,2\dots,r$ if 
\begin{enumerate}
\item $r\geq 13$ when $n\ge 4$,
\item $r\ge 17$ when $n=3$,
\item $r\ge 21$ when $n=2$.
\end{enumerate}
\end{enumerate}
\end{theorem}
We close this section by remarking that Proposition \ref{prop:generic} (resp. Theorem \ref{thm:generic}) only hold for generic subspaces $W_1,\dots, W_r$ of $\mathbb{C}^{n\times n}$, i.e., there is a proper algebraic subvariety $Z_i\subset \operatorname{Gr}(k_i,n^2)$ such that if 
\[
(W_1,\dots,W_r)\in (\operatorname{Gr}(k_1,n^2)-Z_1)\times \cdots \times (\operatorname{Gr}(k_r,n^2)-  Z_r),
\]
then $\rho_r: W_1\times \cdots\times W_r\to \mathbb{C}^2$ is dominant (resp. surjective). However, we do not know any information about algebraic subvarieties $Z_i$. The main contribution of Proposition \ref{prop:generic} and Theorem \ref{thm:generic} is that if the matrix decomposition (both dominant and surjective versions) holds for some $W_1,\dots, W_r$ then it also holds for almost all linear subspaces $W'_i$ of dimensions $\dim W_i,i=1,2,\dots, r$, respectively.

\section{Matrix decomposition for nonlinear spaces}\label{sec:nonlinear}
We consider matrix decompositions for non-linear algebraic subvarieties in this section. In \ref{subsec:companion} we discuss the matrix decomposition into the product of companion matrices and in \ref{subsec:generalized Vandermonde} we discuss the matrix decomposition for generalized Vandermonde matrices.

\subsection{Companion decomposition}\label{subsec:companion}
An $n\times n$ companion matrix is a matrix of the form
\[
\begin{bmatrix}
0 & 0 &\cdots & 0 &c_1\\
1 & 0 &\cdots & 0 &c_2\\
\vdots & \vdots & \ddots&\vdots & \vdots\\
0 & 0 &\cdots & 1 & c_n
\end{bmatrix},
\]
where $c_1,\dots, c_n$ are arbitrary complex numbers. We denote $C_n$ by the set of all companion matrices. Then it is clear that $C_n$ is an affine varitey of dimension $n$. 
\begin{proposition}\label{prop:companion}
A generic $n\times n$ matrix is a product of $n$ companion matrices
\end{proposition}
\begin{proof}
We need to prove that the map 
\[
\rho_n: \underset{n\text{ copies }}{\underbrace{C_n\times \cdots \times C_n}}\to \mathbb{C}^{n\times n}
\]
is dominant. Let $\sigma$ be the matrix corresponding to the permutation $(12\dots n)$, i.e., $\sigma$ is the matrix
\[
\begin{bmatrix}
0 & 0 &\cdots &0 & 1\\
1 & 0 &\cdots &0 & 0\\
\vdots & \vdots & \ddots &\vdots & \vdots \\
0 & 0 & \cdots & 1 & 0
\end{bmatrix}.
\]
For an $n\times n$ matrix $A$ the matrix $\sigma A$ is obtained by shifting the $i$-th row of $A$ to the $(i+1)$-th row, $i=1,\dots, n$. Similarly, the matrix $A\sigma$ is obtained by shifting the $i$-th column of $A$ to the $(i-1)$-th column, $i=1,\dots, n$. Here we adopt the convention that the $n+1$-th row is actually the first row and the $0$-th column is actually the $n$-th column.

We calculate the rank of $d\rho_n$ at the point $(\sigma,\cdots, \sigma)$. First notice that the tangent space $T_{\sigma}C_n$ of $C_n$ at $\sigma$ is the linear space consisting of matrices of the form
\[
\begin{bmatrix}
0 & 0 & \cdots & c_1\\
0 & 0 & \cdots & c_2\\
\vdots & \vdots & \ddots & \vdots\\
0 & 0 & \cdots & c_n\\
\end{bmatrix},\]
where $c_1,\dots, c_n$ are arbitrary complex numbers. Let $Y_1,\cdots, Y_n$ be $n$ elements of $T_{\sigma}C_n$ then by formula \eqref{eqn:differential} we have
\[
d\rho_{|(\sigma,\cdots,\sigma)}(Y_1,\cdots,Y_n)=\sum_{i=1}^n \sigma^{i-1} Y_i \sigma^{n-i}.
\]
Since $\sigma$ corresponds to $(12\cdots n)$, it is easy to see that $\sigma^{i-1}Y_i\sigma^{n-i}$ is a matrix with zero entries everywhere except the $i$-th column. On the other hand, $Y_i$'s are independent from each other, this suffices to show that the rank of $d\rho_{|(\sigma,\cdots,\sigma)}$ is $n^2$. 
\end{proof}

Proposition \ref{prop:surjective} together with Proposition \ref{prop:companion} imply 
\begin{theorem}\label{thm:companion}
Every $n\times n$ invertible matrix is a product of $2n$ companion matrices. Every $n\times n$ matrix is a prodcut of $4n$ companion matrices and a diagonal matrix.
\end{theorem}

Since the map 
\[
\rho_n: \underset{n\text{ copies }}{\underbrace{C_n\times \cdots \times C_n}}\to \mathbb{C}^{n\times n}
\]
is dominant, by Proposition \ref{prop:dimension counting} we see that for a generic matrix $A\in \mathbb{C}^{n\times n}$, the fiber $\rho_n^{-1}(A)$ is of dimension zero and hence $\rho_n^{-1}(A)$ is a finite set. In fact, we can prove more.

\begin{theorem}\label{thm:companion1}
The decomposition of a generic $n\times n$ matrix into the product of $n$ companion matrices is unique, i.e., for a generic $n\times n$ matrix $A$, if
\[
A = C_1\cdots C_n = C'_1 \cdots C'_n,
\] 
where $C_i,C'_i,i=1,2,\dots, n$ are companion matrices then $C_i = C'_i $ for all $i=1,\dots, n$.
\end{theorem}
\begin{proof}
We consider $n$ companion matrices $C_1,\dots, C_n$ and write 
\[
C_i \coloneqq \begin{bmatrix}
0 & 0 &\cdots &0 & c_{i,1}\\
1 & 0 &\cdots &0 & c_{i,2}\\
\vdots & \vdots & \ddots &\vdots & \vdots \\
0 & 0 & \cdots & 1 & c_{i,n}
\end{bmatrix}
\]
and calculate the product 
\[
X^{k} = C_1\dots C_k.
\]
We claim that the $(p,q)$-th entry $X^k_{p,q}$ of $X^k$ is a polynomial in $c_{ij}$ where $1\le i \le n$ and $1\le j \le q-1$ and it is of the form
\[
X^k_{p,q}= \begin{cases}
\sum_{j=1}^{q+k-n-1} X^k_{p,q-j} c_{q+k-n,n-j+1} + c_{q+k-n, n+1+p-q-k}, &~\text{if} q\ge n-k+1\\
1,&~\text{if}~p-q=k~\text{and}~q< n-k+1\\
0,&~\text{otherwise}.
\end{cases}
\]
If $k=n$ then 
\[
X^n_{p,q}=
\sum_{j=1}^{q-1} X^n_{p,q-j} c_{q,n-j+1} + c_{q, p-q+1}.
\]
Now given a generic $n\times n$ matrix $A = (a_{ij})$, we can find $c_{ij},1\le i,j\le n$ such that 
\begin{equation}
\begin{aligned}
a_{p,1} & = c_{1,p},\\
a_{p,2} & = a_{p,1}c_{2,n} + c_{2,p-1},\\
a_{p,3} & = a_{p,2}c_{3,n} + a_{p,1}c_{3,n-1} + c_{3,p-2},\\
& \vdots \\
a_{p,n} & = a_{p,n-1}c_{n,n} + a_{p,n-2}c_{n,n-1} + \cdots + a_{p,1}c_{n,2} + c_{n,p-n+1}, 
\end{aligned}
\label{eqn:companion}
\end{equation}
and those $c_{ij}$'s are uniquely determined by \eqref{eqn:companion}. Those $C_i$ where 
\[
C_i \coloneqq \begin{bmatrix}
0 & 0 &\cdots &0 & c_{i,1}\\
1 & 0 &\cdots &0 & c_{i,2}\\
\vdots & \vdots & \ddots &\vdots & \vdots \\
0 & 0 & \cdots & 1 & c_{i,n}
\end{bmatrix}
\]
are the desired companion matrices such that 
\[
A = C_1 \cdots C_n.
\]
\end{proof}
The proof of Theorem \ref{thm:companion1} actually gives another way to show that a generic $n\times n$ matrix is a product of $n$ companion matrices. Moreover, it also gives an algorithm to decompose a generic $n\times n$ matrix into the product of $n$ companion matrices. In Algorithm \ref{algorithm:companion}, the input is an $n\times n$ matrix $A = (a_{ij})$ with entries $a_{ij},1\le i,j\le n$ and the out put is a sequence of $n$ companion matrices $C_1,\dots, C_n$ such that $A = C_1\cdots C_n$ if such a decomposition exists and is unique.
\begin{algorithm}
  \caption{Companion matrix decomposition}
    \label{algorithm:companion}
  \begin{algorithmic}[1]
  \For{$p,q=1,2,\dots, n$}
  \State Solve the linear system 
  \[a_{p,q}=\sum_{j=1}^{q-1} a_{p,q-j} c_{q,n-j+1} + c_{q, p-q+1}\]
  for $c_{q,1},\dots, c_{q,n}$. Here we adopt the convention that $a_{i,j}, c_{i,j}= 0$ if either $i \le 0$ or $j\le 0$.
  \State If the solution does not exist or is not unique, the decomposition of $A$ does not exist or is not unique. Stop the algorithm. Otherwise, define a matrix
  \[
C_q \coloneqq \begin{bmatrix}
0 & 0 &\cdots &0 & c_{q,1}\\
1 & 0 &\cdots &0 & c_{q,2}\\
\vdots & \vdots & \ddots &\vdots & \vdots \\
0 & 0 & \cdots & 1 & c_{q,n}
\end{bmatrix},
  \]
  and continue the algorithm.
  \EndFor
  \end{algorithmic}
\end{algorithm}

Lastly, we remark that it is not true that every $n\times n$ matrix is a product of $n$ companion matrices. Indeed, if we consider a matrix $A = (a_{ij})$ where 
\[
a_{11} = 0, a_{12} = 1, 
\]
then from \eqref{eqn:companion} that we must have 
\[
c_{11} = a_{11} = 0, a_{11}c_{2n} = a_{12} = 1,
\]
which is impossible. Hence the companion matrix decomposition is an example where the map $\rho_r$ is dominant but not surjective, as we have remarked in Section \ref{sec:toy example}.

\subsection{Generalized Vandermond decomposition}\label{subsec:generalized Vandermonde}
Now we consider generalized Vandermonde matrices. First we need to define generalized Vandermonde matrices.
\begin{definition}
Let $s$ be an integer we call a matrix of the form 
\[ 
\left( (x_q)^{s+p-1}\right)=
\begin{bmatrix}
x_1^s & x_2^s & \cdots & x_{n-1}^s & x_n^s\\
x_1^{s+1} & x_2^{s+1} & \cdots & x_{n-1}^{s+1} &x_n^{s+1}\\
\vdots & \vdots &\ddots&\vdots&\vdots\\
x_1^{s+n-1} & x_2^{s+n-1} & \cdots & x_{s+n-1}^{n-1} &  x_n^{s+n-1} 
\end{bmatrix}
\]
a generalized Vandermonde matrix of type $s$. We denote the set of all generalized Vandermonde matrix of type $s$ by $\operatorname{Vand}_s$ and the set of transpose of generalized Vandermonde matrices of type $s$ by $\operatorname{Vand}^\mathsf{T}_s$.
\end{definition}

By the definition, a Vandermonde matrix is a generalized Vandermonde matrix of type $0$. We consider the matrix decomposition for generalized Vandermonde matrices and their transpose.

\begin{proposition}\label{prop:generalized Vandermonde}
Let $s_1,s_2,\dots, s_n$ be $n$ integers such that $s_i\not \equiv s_j (\textit{ mod }n )$ if $i\ne j$ and $\displaystyle \sum_{j=1}^n s_j\ne 0$. A generic $n\times n$ matrix is a product of elements in $\operatorname{Vand}_{s_i}$ and $\operatorname{Vand}^{\mathsf{T}}_{s_i},i=1,2,\dots,n$.
\end{proposition}
\begin{proof}
Again, we need to prove that the map 
\[
\rho_{2n}:\operatorname{Vand}^T_{s_1}\times \operatorname{Vand}_{s_1}\times \cdots \times \operatorname{Vand}^T_{s_n}\times \operatorname{Vand}_{s_n}\to \mathbb{C}^{n\times n}
\]
is dominant. Let $w$ be a primitive $n$-th root of unity. Let 
\[
W_{i}=\{B_i\cdot A_i | B_i\in \operatorname{Vand}_{s_i}^T, A_i=(w^{-q(p-1+s_i)})\in \operatorname{Vand}_{s_i}\}.
\] 
It is clear that $W_i\subset \operatorname{Vand}^T_{s_i}\cdot \operatorname{Vand}_{s_i}$ and that $\operatorname{Id}\in W_i$. Then it suffices to show that 
\[
\rho_{n}:W_1\times \cdots \times W_n\to \mathbb{C}^{n\times n}
\] 
is dominant. For this, we will show that the differential $d\rho_{n}$ at $(\operatorname{Id},\operatorname{Id},\dots,\operatorname{Id})$ is surjective. Consider the differential
\[
d\rho_{n}|_{(\operatorname{Id},\operatorname{Id},\dots,\operatorname{Id})}(X_1,\dots, X_n)=\displaystyle\sum_{i=1}^n X_i,
\]
where $X_i=((q+s_i-1)w^{p(q+s_i-2)}x _{p,i})\cdot (w^{-q(p+s_i-1)})\in T_{\operatorname{Id}}W_i$ and $x_{p,i}$'s are variables. Then a simple calculation shows that 
\[
X_i=\displaystyle (\sum_{k=1}^n (k+s_i-1)w^{(p-q)k)} )w^{p(s_i-2)-q(s_i-1)}x_{p,s_i}.
\]
Let $\xi_{p,q,i}=\displaystyle (\sum_{k=1}^n (k+s_i-1)w^{(p-q)k)} )w^{p(s_i-2)-q(s_i-1)}$. We regard $\mathbb{C}^{n\times n} $ as $\mathbb{C}^{n^2}$ by the linear isomorphism
\[
h:\mathbb{C}^{n\times n}\to \mathbb{C}^{n^2}
\]
defined by 
\[h((x_{i,j}))=(x_{1,1},x_{1,2},\cdots, x_{1,n}, x_{2,1},x_{2,2},\cdots, x_{2,n},\cdots, x_{n,1},x_{n,2},\cdots, x_{n,n}).\]
Let $M$ be the coefficient matrix of $d\rho_{2n}|_{(\operatorname{Id},\operatorname{Id},\dots,\operatorname{Id})}$ then $M$ is an $n^2\times n^2$ matrix and 
\[M=
\begin{bmatrix}
\xi_{1,1,1} & 0 &\cdots& 0 &\xi_{1,1,2}&0 &\cdots &0&\cdots &\xi_{1,1,n} &0 &\cdots &0\\
\xi_{1,2,1} & 0 &\cdots& 0 &\xi_{1,2,2}&0 &\cdots &0&\cdots &\xi_{1,2,n} &0 &\cdots &0\\
\vdots& \vdots & \ddots&\vdots &\vdots &\vdots  &\ddots & \vdots &\ddots&\vdots&\vdots&\ddots&\vdots\\
\xi_{1,n,1} & 0 &\cdots& 0 &\xi_{1,n,2}&0 &\cdots &0&\cdots &\xi_{1,n,n} &0 &\cdots &0\\

0&\xi_{2,1,1}&\cdots& 0 &0&\xi_{2,1,2} &\cdots &0&\cdots &0&\xi_{2,1,n} &\cdots &0\\
0&\xi_{2,2,1} &\cdots &0 &0&\xi_{2,2,2}&\cdots &0&\cdots &0&\xi_{2,2,n} &\cdots &0\\
\vdots&\vdots & \ddots &\vdots &\vdots  &\vdots & \ddots &\vdots&\ddots&\vdots&\vdots&\ddots&\vdots\\
0&\xi_{2,n,1} &\cdots &0 &0&\xi_{2,n,2} &\cdots &0&\cdots &0&\xi_{2,n,n} &\cdots &0\\

\vdots&\vdots & \ddots &\vdots &\vdots  &\vdots & \ddots &\vdots&\ddots&\vdots&\vdots&\ddots&\vdots\\

0 & 0 &\cdots& \xi_{n,1,1} &0&0 &\cdots &\xi_{n,1,2}&\cdots &0 &0 &\cdots &\xi_{n,1,n}\\
0& 0 &\cdots& \xi_{n,2,1}  &0&0 &\cdots &\xi_{n,2,2}&\cdots &0 &0 &\cdots &\xi_{n,2,n}\\
\vdots& \vdots & \ddots&\vdots &\vdots &\vdots  &\ddots & \vdots &\ddots&\vdots&\vdots&\ddots&\vdots\\
0& 0 &\cdots& \xi_{n,n,1}  &0&0 &\cdots &\xi_{n,n,2}&\cdots &0 &0 &\cdots &\xi_{n,n,n}
\end{bmatrix}.
\] Where $d\rho_{2n}|_{(\operatorname{Id},\operatorname{Id},\dots,\operatorname{Id})}(X_1,\dots, X_n)=\displaystyle\sum_{i=1}^n X_i$ can be expressed as 
\[
\begin{bmatrix}
\xi_{1,1,1} & 0 &\cdots& 0 &\xi_{1,1,2}&0 &\cdots &0&\cdots &\xi_{1,1,n} &0 &\cdots &0\\
\xi_{1,2,1} & 0 &\cdots& 0 &\xi_{1,2,2}&0 &\cdots &0&\cdots &\xi_{1,2,n} &0 &\cdots &0\\
\vdots& \vdots & \ddots&\vdots &\vdots &\vdots  &\ddots & \vdots &\ddots&\vdots&\vdots&\ddots&\vdots\\
\xi_{1,n,1} & 0 &\cdots& 0 &\xi_{1,n,2}&0 &\cdots &0&\cdots &\xi_{1,n,n} &0 &\cdots &0\\

0&\xi_{2,1,1}&\cdots& 0 &0&\xi_{2,1,2} &\cdots &0&\cdots &0&\xi_{2,1,n} &\cdots &0\\
0&\xi_{2,2,1} &\cdots &0 &0&\xi_{2,2,2}&\cdots &0&\cdots &0&\xi_{2,2,n} &\cdots &0\\
\vdots&\vdots & \ddots &\vdots &\vdots  &\vdots & \ddots &\vdots&\ddots&\vdots&\vdots&\ddots&\vdots\\
0&\xi_{2,n,1} &\cdots &0 &0&\xi_{2,n,2} &\cdots &0&\cdots &0&\xi_{2,n,n} &\cdots &0\\

\vdots&\vdots & \ddots &\vdots &\vdots  &\vdots & \ddots &\vdots&\ddots&\vdots&\vdots&\ddots&\vdots\\

0 & 0 &\cdots& \xi_{n,1,1} &0&0 &\cdots &\xi_{n,1,2}&\cdots &0 &0 &\cdots &\xi_{n,1,n}\\
0& 0 &\cdots& \xi_{n,2,1}  &0&0 &\cdots &\xi_{n,2,2}&\cdots &0 &0 &\cdots &\xi_{n,2,n}\\
\vdots& \vdots & \ddots&\vdots &\vdots &\vdots  &\ddots & \vdots &\ddots&\vdots&\vdots&\ddots&\vdots\\
0& 0 &\cdots& \xi_{n,n,1}  &0&0 &\cdots &\xi_{n,n,2}&\cdots &0 &0 &\cdots &\xi_{n,n,n}
\end{bmatrix}
\cdot \begin{bmatrix}x_{1,s_1}\\ x_{2,s_1}\\ \vdots\\ x_{n,s_1}\\ x_{1,s_2}\\ x_{2,s_2}\\ \cdots \\x_{n,s_2}\\ \vdots \\  x_{1,s_n}\\ x_{2,s_n}\\ \vdots\\ x_{n,s_n} \end{bmatrix}.
\]
Now we will prove that $M$ is of the full rank $n^2$. Since by permutation of columns, $M$ becomes a block diagonal matrix where blocks $M_p$ on the diagonal are 
\[M_p=
\begin{bmatrix}
\xi_{p,1,1}& \xi_{p,1,2}&\cdots &\xi_{p,1,n}\\
\xi_{p,2,1}& \xi_{p,2,2}&\cdots &\xi_{p,2,n}\\
\vdots &\vdots &\ddots&\vdots\\
\xi_{p,n,1}& \xi_{p,n,2}&\cdots &\xi_{p,n,n}
\end{bmatrix}.
\]
Therefore, it suffices to prove that $M_p$ is of rank $n$ for each $p=1,2,\dots, n$. Note that we have a formulas
\begin{align*}
\displaystyle \sum_{k=1}^n k w^{k}&=\frac{-nw}{1-w},\\
\displaystyle \sum_{k=1}^n w^k&=0.
\end{align*}

for any $n$-th root of unity $w\ne 1$. Hence 
\[
\xi_{p,i,j}=\begin{cases}
-\frac{nw^{p-i}}{1-w^{p-i}} w^{(p-i)s_j-2p+i},&\textit{ if }p\ne i\\
\frac{(2s_j+n-1)nw^{-p}}{2},&\textit{ if }p=i.
\end{cases}
\]
Then it is easy to see that the rank of $M_p$ is the same as the matrix $\tilde{M_p}=(\tilde{\xi_{p,i,j}})$, where 
\[
\tilde{\xi}_{p,i,j}=\begin{cases}
w^{-is_j},&\textit{ if }i\ne p\\
s_jw^{-ps_j},&\textit{ if }i=p.
\end{cases}
\]
Up to a permutation of rows, $\tilde{M_p}$ is of the form  
\[
\begin{bmatrix}
1 & 1 & \cdots & 1\\
w^{-s_1} & w^{-s_2}&\cdots & w^{-s_n}\\
\vdots & \vdots &\ddots & \vdots\\
w^{-(p-1)s_1} & w^{-(p-1)s_2}&\cdots & w^{-(p-1)s_n}\\
\alpha_1w^{-ps_1} & \alpha_2w^{-ps_2}&\cdots & \alpha_nw^{-ps_n}\\
w^{-(p+1)s_1} & w^{-(p+1)s_2}&\cdots & w^{-(p+1)s_n}\\
\vdots & \vdots &\ddots & \vdots\\
w^{-(n-1)s_1} & w^{-(n-1)s_2}&\cdots & w^{-(n-1)s_n}
\end{bmatrix},
\] 
where $\alpha_1,\cdots, \alpha_n$ are fixed integers. We still denote this matrix by $\tilde{M_p}$ where $p=0,1,\cdots, n-1$. We will compute the determinant of $\tilde{M_p}$. We have 
\[
\operatorname{Det}(\tilde{M_p})=\displaystyle \sum_{j=1}^n (-1)^{p+j-1}\alpha_j w^{-ps_j}V(w^{-s_j})S_{j,n-1-p}.
\]
Where $V(w^{-s_j})$ is the determinant of the Vandermonde matrix determined by $w^{-s_1}$, $\cdots$, $w^{-s_{j-1}}$, $w^{-s_{j+1}}$, $\cdots, w^{-s_n}$ and $S_{j,n-1-p}$ is the $(n-1-p)$-th symmetric function on $w^{-s_1}$, $\cdots$, $w^{-s_{j-1}}$, $w^{-s_{j+1}}$, $\cdots$, $w^{-s_n}$. 

Note that 
\[
(t-w^{-s_1})\cdots (t-w^{-s_n})=t^n-1,
\]
so 
\begin{align*}
(t-w^{-s_1})\cdots (t-w^{-s_{j-1}}) (t-w^{-s_{j+1}})\cdots (t-w^{-s_n})&=\frac{t^n-1}{t-w^{-s_j}}\\
&=\displaystyle \sum_{k=0}^{n-1} (w^{-s_j})^{k}t^{n-1-k},
\end{align*}
hence $S_{j,n-1-p}=(-1)^p(w^{-s_j})^{n-1-p}=(w^{-s_j})^{-1-p}$.

On the other hand, we know that 
\[
V(w^{-s_j})=\displaystyle \prod_{a> b, a,b\ne j} (w^{-s_{a}}-w^{-s_{b}})=(-1)^{j-1}\frac{\displaystyle\prod_{a>b} (w^{-s_a}-w^{s_b})}{\displaystyle\prod_{c\ne j} (w^{-s_c}-w^{-s_j})}.
\]
Let $V$ be the determinant of the Vandermonde matrix determined by $w^{-s_1},\cdots, w^{-s_n}$. It is obvious that $V\ne 0$. Also from
\[
\frac{t^n-1}{t-w^{-s_j}}=\displaystyle \prod_{c\ne j} (t-w^{-s_c}),
\]
we have 
\[
\displaystyle\prod_{c\ne j} (w^{-s_c}-w^{-s_j})=\displaystyle \lim_{t\to w^{-s_j}}\frac{t^n-1}{t-w^{-s_j}}=nw^{s_j}.
\]

Therefore 
\begin{align*}
\operatorname{Det}(\tilde{M_p})&=\displaystyle \sum_{j=1}^n (-1)^{p+j-1}\alpha_j w^{-ps_j} (-1)^p (w^{-s_j})^{-1-p}(-1)^{j-1}\frac{\displaystyle\prod_{a>b} (w^{-s_a}-w^{s_b})}{\displaystyle\prod_{c\ne j} (w^{-s_c}-w^{-s_j})}\\
&= V\displaystyle \sum_{j=1}^n \frac{\alpha_j w^{s_j}}{\displaystyle\prod_{c\ne j} (w^{-s_c}-w^{-s_j})}\\
&= V \displaystyle \sum _{j=1}^n\frac{\alpha_jw^{s_j}}{nw^{s_j}}\\
&=\frac{V}{n} \displaystyle \sum_{j=1}^n \alpha_j.
\end{align*}
In particular, we set $\alpha_j=s_j$ then we see that $\operatorname{Det}(\tilde{M_p})\ne 0$ if and only if $\displaystyle \sum_{j=1}s_j\ne 0$. This implies that the map $\rho_{n}$ is dominant for all $s_1,\dots, s_n$ such that $s_i\not\equiv s_j (\textit{ mod }n)$ and $\displaystyle \sum_{j=1}^ns_j\ne 0$.
\end{proof}

From the proof of Proposition \ref{prop:generalized Vandermonde} we have 
\begin{corollary}\label{cor:generalized Vandermonde}
Let $s_1,\dots, s_n$ be as in Proposition \ref{prop:generalized Vandermonde} and let
\[
W_{i}=\{B_i\cdot A_i | B_i\in \operatorname{Vand}_{s_i}^T, A_i=(w^{-q(p-1+s_i)})\}.
\]
A generic $n\times n$ matrix is a product of elements in $W_i,i=1,2,\dots, n$.
\end{corollary}

Combining Proposition \ref{prop:surjective} and Proposition \ref{prop:generalized Vandermonde} we obtain
\begin{theorem}
Let $s_1,\dots, s_n$ be as in Proposition \ref{prop:generalized Vandermonde}. For every $n\times n$ invertible matrix $M$ there are $A_i,A'_i\in \operatorname{Vand}_{s_i}$ and $B_i,B'_i\in \operatorname{Vand}^{\mathsf{T}}_{s_i}$,$i=1,2,\dots, n$ such that 
\[
M = B_1A_1\cdots B_nA_n B'_1A'_1\cdots B'_nA'_n.
\]
For every $n\times n$ matrix $M$, there are $A_i,A'_i,C_i,C'_i\in \operatorname{Vand}_{s_i}$, $B_i,B'_i,D_i,D'_i\in \operatorname{Vand}^{\mathsf{T}}_{s_i}$,$i=1,2,\dots, n$ and a diagonal matrix $E$ such that 
\[
M = B_1A_1\cdots B_nA_n B'_1A'_1\cdots B'_nA'_n E D_1C_1\cdots D_nC_n D'_1C'_1\cdots D'_nC'_n.
\]
\end{theorem}

\begin{theorem}
Let $W_i,i=1,2\dots, n$ be as in Corollary \ref{cor:generalized Vandermonde}. For every $n\times n$ invertible matrix $M$ there is an element $A_i$ in $W_i$ for each $i=1,2,\dots, n$ such that 
\[
M = A_1\cdots A_n.
\]
For every $n\times n$ matrix $M$ there are $A_i,B_i\in W_i$ for each $i=1,2,\dots,n$ and a diagonal matrix $C$ such that
\[
M = A_1\cdots A_nCB_1\cdots B_n.
\] 
\end{theorem}

\section{Conclusion}
We discuss the existence of matrix decompositions for bidiagonal, skew symmetric, symmetric Toeplitz, persymmetric Hankel, generic, companion, generalized Vandermonde matrix decompositions, for both generic and arbitrary matrices. 

It is natural to ask, for example, if the number of bidiagonal matrices needed to decompose a generic (resp. arbitrary) matrix is the smallest. For most types of matrix decompositions we discussed in this paper, the number we obtain is already the smallest for a generic matrix. It is still open if the number  we obtain is the smallest for an arbitrary matrix. We summarize our main results in the following table.
\begin{center}
  \begin{tabular}{ l | c | c | c | c | r}
    \hline
    &   $r$ (generic) &  sharpness & $r$ (arbitrary)  & algorithm\\  \hline
    bidiagonal & $2n$ &   unknown &  $8n$   & unknown \\ \hline
    skew symmetric ($n\ge 8$ even) & $3$ &    yes & 13  & unknown\\ \hline
    symmetric Toeplitz & $\lfloor \frac{n+1}{2} \rfloor$ & yes  & unknown & unknown\\ \hline
    companion &  $n$ &  yes & $4n+1$  & yes \\ \hline
    generalized Vandermonde &  $2n$ & unknown & $8n+1$  & unknown\\ \hline
      \end{tabular}
\end{center}


\begin{thebibliography}{99}
\bibitem {10}``Algorithms for the ages,'' \textit{Science}, \textbf{287}
(2000), no.~5454, pp.~799.

\bibitem {Borel}A.~Borel, \textit{Linear Algebraic Groups}, 2nd Ed., Graduate
Texts in Mathematics, \textbf{126}, Springer-Verlag, New York, NY, 1991.

\bibitem {Bosch}A.~J.~Bosch, ``The factorization of a square matrix into two
symmetric matrices,'' \textit{Amer.\ Math.\ Monthly}, \textbf{93} (1986),
no.~6, pp.~462--464.

\bibitem {Howe}R.~Howe, ``Very basic Lie theory,''
\textit{Amer.\ Math.\ Monthly}, \textbf{90} (1983), no.~9, pp.~600--623.

\bibitem {Knapp}A.~W.~Knapp, \textit{Lie Groups Beyond an Introduction}, 2nd
Ed., Progress in Mathematics, \textbf{140}, Birkh\"{a}user, Boston, MA, 2002.

\bibitem {Decomp}G.~W.~Stewart, ``The decompositional approach to matrix
computation,'' \textit{Comput.\ Sci.\ Eng.}, \textbf{2} (2000), no.~1, pp.~50--59.

\bibitem {Taylor}J.~L.~Taylor, \emph{Several Complex Variables with
Connections to Algebraic Geometry and Lie groups}, Graduate Studies in
Mathematics, \textbf{46}, American Mathematical Society, Providence, RI, 2002.

\bibitem{yelim}K.~Ye and L.-H.~Lim, ``Every matrix is a product of Toeplitz matrices,'' \textit{Found. Comput. Math.}, \textbf{15} (2015), no.~6, pp.~1-22.

\bibitem{shafarevich} I.~R.~Shafarevich, ``Basic algebraic geometry. 1", Springer, Heidelberg, 2013.


\bibitem{BA} R.~B.~Bitmead and B.~D.~O.~Anderson, ``Asymptotically fast solution of Toeplitz and related systems of linear
equations", \textit{Linear Algebra Appl.}, \textbf{34} (1980), pp.~103-116.

\bibitem{Weyman} J.~Weyman, \textit{Cohomology of vector bundles and syzygies}, Cambridge Tracts in Mathematics, \textbf{149}, Cambridge University Press, Cambridge, 2003.

\bibitem{Harris} J.~Harris, \textit{Algebraic geometry}, Graduate Texts in Mathematics, \textbf{133}, Springer-Verlag, New York, 1995.

\bibitem{GH} P.~Griffiths and J.~Harris, \textit{Principles of algebraic geometry}, Wiley Classics Library, John Wiley \& Sons, Inc., New York, 1994.

\bibitem{FH} W.~Fulton and J.~Harris, \textit{Representation theory}, Graduate Texts in Mathematics, \textbf{129}, A first course,
              Readings in Mathematics, Springer-Verlag, New York, 1991.
              
              \bibitem{fast2} G.~S.~Ammar and W.~B.~Gragg, ``Superfast solution of real positive definite Toeplitz systems,'' \textit{SIAM J.\ Matrix Anal.\ Appl.}, \textbf{9} (1988), no.~1, pp.~61--76.

\bibitem{ref1} Chandrasekaran, Venkat and Sanghavi, Sujay and Parrilo, A. Pablo and Willsky, S. Alan, ``Rank-sparsity incoherence for matrix decomposition", \textit{SIAM J. Optim.}, \textbf{21} (2011), no.~2, pp.~572-596.

\bibitem{ref2} D.~L.~Donoho, and X.~M.~Huo, ``Uncertainty principles and ideal atomic decomposition", \textit{IEEE Trans. Inform. Theory}, \textbf{47} (2001), no.~7, pp.~ 2845-2862.

\bibitem{ref3} J.~B.~Victor, Z. Douglas and C. David, ``Low-rank network decomposition reveals structural characteristics of small-world networks", \textit{Phys. Rev. E}, vol.~92 (2015), iss.~6, 062822.

\bibitem{Hartshorne} H.~Robin, \textit{Algebraic Geometry}, \textbf{52}, Graduate Texts in Mathematics, Springer-Verlag, New York-Heidelberg, 1977.

\bibitem{RJ} H.~Roger, C.~R.~Johnson, \textit{Matrix Analysis}, Cambridge University Press, 1985.

\bibitem{O} P.~Okunev, ``Necessary and sufficient conditions for existence of the LU factorization of an arbitrary matrix", \textit{arXiv: 0506382}, (1997).

\bibitem{M2} Grayson, R.~Daniel and M.~E.~Stillman, \textit{Macaulay2, a software system for research in algebraic geometry}, Available at \url{http://www.math.uiuc.edu/Macaulay2/}.

\end{thebibliography}
\end{document}